\documentclass[11pt, letterpaper]{article}

\pdfoutput=1

\usepackage{amsmath}
\usepackage{amsfonts}
\usepackage{amssymb} 
\usepackage{amsthm}
\usepackage{mathrsfs} 
\usepackage{mathtools} 
\usepackage[nottoc,numbib]{tocbibind} 
\usepackage[usenames,dvipsnames,svgnames,table]{xcolor} 

\usepackage{tikz}
\usetikzlibrary{shapes,matrix,arrows,chains,trees,calc} 

\usepackage{xr} 
\usepackage{cite} 
\usepackage{enumerate} 
\usepackage{url}
\usepackage[colorlinks=true, allcolors=MidnightBlue, pdftitle={5-dimensional geometries II: the non-fibered geometries}, pdfauthor={Andrew Geng}]{hyperref}
\usepackage[total={6.5in,9in}]{geometry}
\usepackage[T1]{fontenc} 
\usepackage[utf8]{inputenc} 

\title{5-dimensional geometries II: the non-fibered geometries}
\author{Andrew Geng}

\theoremstyle{plain}
\newtheorem{thm}{Theorem}[section]
\newtheorem{prop}[thm]{Proposition}

\newtheorem{lemma}[thm]{Lemma}
\newtheorem{cor}[thm]{Corollary}
\theoremstyle{definition}
\newtheorem{defn}[thm]{Definition}

\theoremstyle{remark}
\newtheorem{rmk}[thm]{Remark}
\newtheorem{eg}[thm]{Example}
\makeatletter
\let\c@figure\c@thm
\let\c@table\c@thm
\makeatother
\numberwithin{figure}{section}
\numberwithin{table}{section}

\newcommand{\keyword}{\emph}
\newcommand{\op}{\operatorname}
\newcommand{\ol}{\overline}
\newcommand{\R}{\mathbb{R}}
\newcommand{\C}{\mathbb{C}}
\newcommand{\Q}{\mathbb{Q}}

\newcommand{\Z}{\mathbb{Z}}
\newcommand{\Hyp}{\mathbb{H}}
\newcommand{\Euc}{\mathbb{E}}
\newcommand{\Heis}{\mathrm{Heis}}
\newcommand{\Sol}{\mathrm{Sol}}
\newcommand{\Nil}{\mathrm{Nil}}

\newcommand{\lie}{\mathfrak}

\newcommand{\vecspan}[1]{ \operatorname{span}\left( { { #1 } } \right) } 
\newcommand{\semisum}{\mathrlap{+}{\supset}} 

\newcommand{\nilrad}{\op{nil}}
\newcommand{\SemiR}[2]{ { \underset{ {#2} }{ { {#1} } \rtimes \R} } }

\begin{document}

\maketitle

\begin{abstract}
We classify the $5$-dimensional homogeneous geometries in the sense of Thurston.
The present paper (part 2 of 3) classifies those in which
the linear isotropy representation is either irreducible or trivial.
The $5$-dimensional geometries with irreducible isotropy are
the irreducible Riemannian symmetric spaces,
while those with trivial isotropy
are simply-connected solvable Lie groups of the form
$\R^3 \rtimes \R^2$ or $N \rtimes \R$ where $N$ is nilpotent.
\end{abstract}

\tableofcontents

\section{Introduction}

Riemannian homogeneous spaces appeared in Thurston's Geometrization Conjecture
as local models for pieces in a decomposition of $3$-manifolds.
Those with compact quotients and maximal symmetry,
the eight \keyword{geometries} (see Defn.~\ref{defn:geometries} for details),
are classified by Thurston in \cite[Thm.~3.8.4]{thurstonbook}.

Filipkiewicz classified the $4$-dimensional geometries in \cite{filipk},
retaining the conditions that make a homogeneous space a geometry in the
sense of Thurston. Imitating this approach in the $5$-dimensional case,
the problem divides into cases for each
representation by which the point stabilizer $G_p$ of $p \in M = G/G_p$
acts on the tangent space $T_p M$ (the ``linear isotropy representation'').

If $T_p M$ decomposes in a nice way into lower-dimensional
irreducible sub-representations, then this decomposition can
be exploited to classify the resulting $5$-dimensional geometries,
via a $G$-invariant fiber bundle structure on $M$.
Such an approach is carried out in Part III \cite{geng3}.
The present paper concerns itself with the classification
when $T_p M$ is irreducible or trivial; in this case,
$T_p M$ has no proper characteristic summands as an $G_p$-representation,
which makes its decomposition less useful.

Instead, we appeal to existing classifications:
the classification of strongly isotropy irreducible homogeneous spaces by
Manturov \cite{manturov1, manturov2, manturov3, manturov1998},
Wolf \cite{wolf_irr, wolf_irr_fix}, and Kr\"amer \cite{kramer1975}
when $T_p M$ is irreducible;
and the method used by Mubarakzyanov in \cite{muba_solvable5}
and Filipkiewicz in \cite[Ch.~6]{filipk}
for classifying unimodular solvable real Lie algebras by nilradical
when $T_p M$ is trivial.
Adapting these classifications to the setting of Thurston's geometries
requires answering questions about lattices and maximal isometry groups,
which is done in this paper.
The result is as follows.
\begin{thm}[\textbf{Classification of $5$-dimensional maximal model
    geometries with irreducible or trivial isotropy}] \label{thm:main}

    Let $M = G/G_p$ be a $5$-dimensional maximal model geometry.
    \begin{enumerate}[(i)]
        \item If $G_p$ acts irreducibly on the tangent space $T_p M$,
            then $M$ is one of the classical spaces
            $\Euc^5$, $S^5$, and $\Hyp^5$ with its usual isometry group,
            or one of the other irreducible Riemannian
            symmetric spaces $\op{SL}(3,\R)/\op{SO}(3)$
            and $\op{SU}(3)/\op{SO}(3)$. (Prop.~\ref{thm:irr_list})
        \item If $G_p$ acts trivially on $T_p M$, then $M \cong G$ is
            one of the following connected, simply-connected, solvable Lie groups.
            \begin{enumerate}
                \item
                    $\R^3 \rtimes \{xyz=1\}^0$,
                    the semidirect product $\R^3 \rtimes \R^2$ where $\R^2$ acts by
                    diagonal matrices with positive entries and determinant $1$.
                    (Prop.~\ref{prop:split_r3})
                \item
                    $\R^4 \rtimes_{e^{tA}} \R$, where the multiset of
                    characteristic polynomials of the Jordan blocks of $A$
                    is one of the multisets below.
                    These are abbreviated as $\SemiR{\R^4}{\text{polynomials}}$.
                    (Prop.~\ref{prop:semiwithabelian})
                    \begin{itemize}
                        \item $\{x-a,\, x-b,\, x-c,\, x+a+b+c\}$, where
                                \begin{enumerate}[1.]
                                    \item $a \neq b \neq c \neq a$, and
                                    \item $e^{tA}$ has integer characteristic polynomial
                                        for some $t$.
                                \end{enumerate}
                            (This is a countably infinite family of geometries.)
                        \item $\{x^2,\, x-1,\, x+1\}$
                        \item $\{(x-1)^2,\, (x+1)^2\}$
                        \item $\{x^3,\, x\}$
                        \item $\{x^4\}$
                    \end{itemize}
                \item
                    $\SemiR{\Nil^4}{3 \to 1}$ and $\SemiR{\Nil^4}{4 \to 3 \to 1}$,
                    whose Lie algebras have basis $x_1,\ldots,x_5$ and
                    the following nonzero brackets.
                    (Prop.~\ref{seminilfour})
                    \begin{align*}
                        [x_4,x_3] &= x_2
                            & [x_5, x_3] &= x_1 \\
                        [x_4,x_2] &= x_1
                            & [x_5, x_4] &= 0 \text{ or } x_3 \text{ respectively}.
                    \end{align*}
                \item
                    $\R \times \Sol^4_1$ and
                    $\SemiR{(\R \times \Heis_3)}{\text{Lorentz},\, y \to x_1}$,
                    whose Lie algebras have basis $y,x_1,x_2,x_3,z$ and
                    the following nonzero brackets.
                    (Prop.~\ref{prop:heisenpluslorentz})
                    \begin{align*}
                        [x_3,x_2] &= x_1  &
                        [z,x_2] &= x_2  &
                        [z,x_3] &= -x_3  &
                        [z,y] &\mapsto 0 \text{ or } x_1 \text{ respectively}.
                    \end{align*}
            \end{enumerate}
    \end{enumerate}
    Moreover, all of the above spaces
    are maximal model geometries.
\end{thm}

\begin{rmk}
    Properties of the geometries are not explored in depth here.
    For the symmetric spaces in Thm.~\ref{thm:main}(i), one can
    see e.g.\ \cite[\S 9.6]{wolf} for a discussion of $\op{SU}(3)/\op{SO}(3)$
    and \cite[Appendix 5]{ballmangromovschroeder} for
    a discussion of $\op{SL}(3,\R)/\op{SO}(3)$.
    To make sense of the array of solvable Lie groups in
    Thm.~\ref{thm:main}(ii), one can construct an identification key
    to distinguish them, such as Figure \ref{fig:solvable_flowchart}
    (proven in Prop.~\ref{prop:solvable_distinctness}).
\end{rmk}

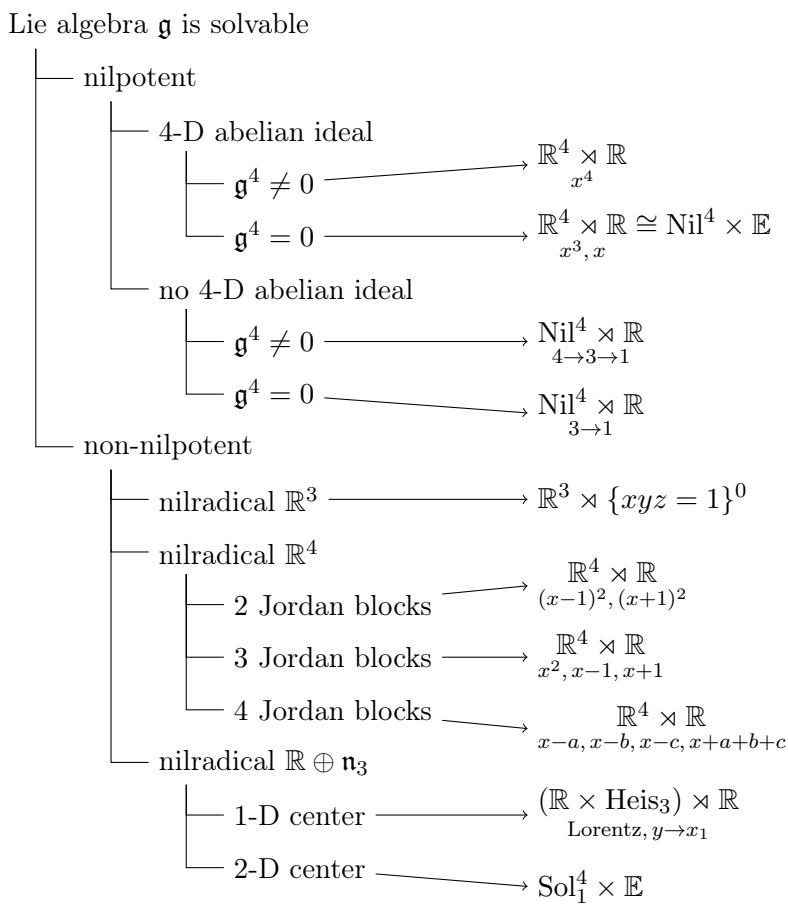
\begin{figure}[h!]
    \caption{Identification key for geometries $G = G/\{1\}$.}
    \label{fig:solvable_flowchart}
    \label{fig_thesis:solvable_flowchart}
    \begin{center}\begin{tikzpicture}[%
            grow via three points={one child at (1.0,-0.7) and
            two children at (1.0,-0.7) and (1.0,-1.4)},
            growth parent anchor=west,
            edge from parent path={($(\tikzparentnode.south west)+(0.5,0)$) |- (\tikzchildnode.west)}]
        \tikzset{
            every node/.style={anchor=west},
            final/.style={draw=none},
        }
        \node {Lie algebra $\lie{g}$ is solvable}
            child { node {nilpotent}
                child { node {$4$-D abelian ideal}
                    child { node (tnil5) {$\lie{g}^4 \neq 0$}}
                    child { node (tnil4xe) {$\lie{g}^4 = 0$}}
                }
                child [missing] {}
                child [missing] {}
                child { node {no $4$-D abelian ideal}
                    child { node (tnil431) {$\lie{g}^4 \neq 0$}}
                    child { node (tnil43) {$\lie{g}^4 = 0$}}
                }
                child [missing] {}
                child [missing] {}
            }
            child [missing] {}
            child [missing] {}
            child [missing] {}
            child [missing] {}
            child [missing] {}
            child [missing] {}
            child { node {non-nilpotent}
                child { node (tr3) {nilradical $\R^3$}}
                child { node {nilradical $\R^4$}
                    child { node (tj2) {$2$ Jordan blocks}}
                    child { node (tj3) {$3$ Jordan blocks}}
                    child { node (tj4) {$4$ Jordan blocks}}
                }
                child [missing] {}
                child [missing] {}
                child [missing] {}
                child { node {nilradical $\R \oplus \lie{n}_3$}
                    child { node (tz1) {$1$-D center}}
                    child { node (tz2) {$2$-D center}}
                }
            };
        \node[right=of tj4] (geoms) {};
        \draw[->] (tr3) -- (geoms |- tr3) node {$\R^3 \rtimes \{xyz=1\}^0$};
        \draw[->] (tnil5) -- ($(geoms |- tnil5)+(0,0.25)$) node {$\SemiR{\R^4}{x^4}$};
        \draw[->] (tnil4xe) -- (geoms |- tnil4xe) node {$\SemiR{\R^4}{x^3,\, x} \cong \Nil^4 \times \Euc$};
        \draw[->] (tnil431) -- (geoms |- tnil431) node {$\SemiR{\Nil^4}{4\to3\to1}$};
        \draw[->] (tnil43) -- ($(geoms |- tnil43)+(0,-0.25)$) node {$\SemiR{\Nil^4}{3\to1}$};
        \draw[->] (tj2) -- ($(geoms |- tj2)+(0,0.25)$) node {$\SemiR{\R^4}{(x-1)^2,\, (x+1)^2}$};
        \draw[->] (tj3) -- (geoms |- tj3) node {$\SemiR{\R^4}{x^2,\, x-1,\, x+1}$};
        \draw[->] (tj4) -- ($(geoms |- tj4)+(0,-0.25)$) node {$\SemiR{\R^4}{x-a,\, x-b,\, x-c,\, x+a+b+c}$};
        \draw[->] (tz1) -- (geoms |- tz1) node {$\SemiR{(\R \times \Heis_3)}{\text{Lorentz},\, y \to x_1}$};
        \draw[->] (tz2) -- ($(geoms |- tz2)+(0,-0.25)$) node {$\Sol^4_1 \times \Euc$};
    \end{tikzpicture}\end{center}
\end{figure}

\begin{rmk}
    When the group is a direct product with $\R$,
    the geometry is written as a product with $\Euc$,
    following Thurston's convention in \cite[Thm.~3.8.4]{thurstonbook}
    to highlight that the $\R$ factor behaves as a $1$-dimensional Euclidean space.

    Some of the $\R^4 \rtimes \R$ geometries
    can also be named as products
    of lower-dimensional geometries, using names from
    \cite[Table 1]{wall86}.
    If the polynomials are $x^3$ and $x$, then the Lie group is isomorphic
    to $\Nil^4 \times \R$.
    If there are $4$ polynomials and one of them is $x$,
    then the Lie group is isomorphic to $\Sol^4_{m,n} \times \R$.

    An alternative naming scheme is given in \cite[Table II]{patera},
    from Mubarakzyanov's classification \cite{muba_solvable5} of solvable Lie algebras;
    Table \ref{table:patera_concordance} shows the concordance.
\end{rmk}

\begin{table}[h!]
    \caption[Names from {\cite[Tables I, II]{patera}} for
        geometries with trivial isotropy.
    ]{Lie algebra names from \cite[Tables I, II]{patera} for
        geometries with trivial isotropy.}
    \label{table:patera_concordance}
    \begin{center}\begin{tabular}{cl}
        Geometry & Lie algebra \\
        \hline
        \rule[0pt]{0pt}{15pt}
        $\Nil^4 \times \Euc$
            & $A_{4,1} \oplus \R$ \\
        $\Sol^4_1 \times \Euc$
            & $A_{4,8} \oplus \R$ \\
        \rule[0pt]{0pt}{15pt}
        $\SemiR{\R^4}{x^4}$
            & $A_{5,2}$ \\
        \rule[0pt]{0pt}{17pt}
        $\SemiR{\Nil^4}{3\to1}$
            & $A_{5,5}$ \\
        \rule[0pt]{0pt}{15pt}
        $\SemiR{\Nil^4}{4\to3\to1}$
            & $A_{5,6}$ \\
        \rule[0pt]{0pt}{15pt}
        $\SemiR{\R^4}{x-a,\, x-b,\, x-c,\, x+a+b+c}$
            & $A^{a,b,c}_{5,7}$ ($a+b+c=-1$) \\
        \rule[0pt]{0pt}{17pt}
        $\SemiR{\R^4}{x^2,\, x-1,\, x+1}$
            & $A^{-1}_{5,8}$ \\
        \rule[0pt]{0pt}{15pt}
        $\SemiR{\R^4}{(x-1)^2,\, (x+1)^2}$
            & $A^{-1}_{5,15}$ \\
        \rule[0pt]{0pt}{15pt}
        $\SemiR{(\R \times \Heis_3)}{\text{Lorentz},\, y \to x_1}$
            & $A^{0}_{5,20}$ \\
        \rule[0pt]{0pt}{15pt}
        $\R^3 \rtimes \{xyz=1\}^0$
            & $A^{-1,-1}_{5,33}$ \\
    \end{tabular}\end{center}
\end{table}

Section \ref{chap:background3} collects some basic facts about geometries and homogeneous spaces.
Section \ref{chap:isotropy} classifies the isotropy representations (point stabilizers).
Sections \ref{chap:irr} and \ref{chap:sol} carry out the classification of geometries
when the isotropy is irreducible or trivial, respectively.

\section{General background: geometries}
\label{chap:background3}

This section recalls terms and basic facts about geometries,
including their interpretations as Riemannian manifolds or homogeneous spaces.
This starts with the definition of a geometry, following Thurston in
\cite[Defn.\ 3.8.1]{thurstonbook} and Filipkiewicz in \cite[\S 1.1]{filipk}.
\begin{defn}[\textbf{Geometries}] \label{defn:geometries} ~
    \begin{enumerate}[(i)]
        \item A \keyword{geometry} is a pair $(M,G \subseteq \op{Diff} M)$ where
            $M$ is a connected, simply-connected smooth manifold
            and $G$ is a connected Lie group
            acting transitively, smoothly, and with compact point stabilizers.
        \item $(M,G)$ is a \keyword{model geometry} if
            there is a finite-volume complete
            Riemannian manifold $N$ locally isometric
            to $M$ with some $G$-invariant metric. Such an $N$
            is said to be \keyword{modeled on} $(M,G)$.
        \item $(M,G)$ is \keyword{maximal} if there is no geometry $(M,H)$
            with $G \subsetneq H$.
            Any such $(M,H)$ is said to \keyword{subsume} $(M,G)$.
    \end{enumerate}
\end{defn}

As Thurston remarks in \cite[\S 3.8]{thurstonbook},
$M$ can be thought of as a Riemannian manifold as long as one
is willing to change the metric. Since the tools for proving
maximality are stated in the language of Riemannian geometry,
here is the explicit relationship with Riemannian manifolds.

\begin{prop}[\textbf{Existence of invariant metric}; see e.g.\ {\cite[Prop.\ 3.4.11]{thurstonbook}}]
    Suppose a Lie group $G$ acts transitively, smoothly, and
    with compact point stabilizers on a smooth manifold $M$. Then
    $M$ has a $G$-invariant Riemannian metric.
\end{prop}

This makes $G$ a subgroup of the isometry group.
The next Proposition asserts this inclusion is an equality for maximal geometries,
which makes it possible to approach questions of maximality
by leveraging existing results describing isometry groups of various spaces.

\begin{prop}[\textbf{Description of maximality}, {\cite[Prop.\ 1.1.2]{filipk}}] \label{prop:maximality_with_metric} ~
    \begin{enumerate}[(i)]
        \item Every geometry is maximal or subsumed by a maximal geometry.
        \item If $(M,G)$ is a maximal geometry, then $G$ is the identity component
            $(\op{Isom} M)^0$ of $\op{Isom} M$ in any $G$-invariant metric on $M$.
    \end{enumerate}
\end{prop}
\begin{rmk} \label{rmk:maximality_by_isotropy}
    By connectedness, if $(M,G)$ is realized by $(M,H)$ with $G \subsetneq H$,
    then $\dim G < \dim H$ and $\dim G_p < \dim H_p$. Hence to verify
    maximality of a geometry $(M,G)$, it suffices to distinguish $(M,G)$
    from the geometries whose point stabilizers contain $G_p$.
    Knowing inclusions between possible point stabilizers
    (Figure \ref{fig:isotropy_poset}) will facilitate this.
\end{rmk}

The interpretation of geometries as homogeneous spaces is already
integral to the $4$-dimensional classification by Filipkiewicz,
and some geometries are most concisely expressed as homogeneous spaces.
So the dictionary is provided explicitly below, with an outline of the proof
and prerequisites in case the reader requires details.
\begin{prop}[\textbf{Geometries described as homogeneous spaces}] \label{prop:geometries3} ~
    \begin{enumerate}[(i)]
        \item Geometries $(M,G)$ correspond one-to-one
            with simply-connected homogeneous spaces $G/G_p$ where $G$ is a connected
            Lie group and $G_p$ is compact and contains no
            nontrivial normal subgroups of $G$.
            The correspondence is $G/G_p \leftrightarrow (G/G_p, G)$.
        \item $G/G_p$ is a model geometry if and only if
            some lattice $\Gamma \subset G$ intersects no conjugate of $G_p$
            nontrivially.
        \item $G/G_p$ is a maximal geometry if and only if
            it is not $G$-equivariantly diffeomorphic to a geometry
            $G'/G'_p$ with $G \subsetneq G'$.
    \end{enumerate}
\end{prop}
\begin{proof}
    Part (i) follows from two standard facts:
    \begin{enumerate}
        \item If a Lie group $G$ acts transitively on a smooth manifold $M$
            with subgroup $G_p$ stabilizing $p \in M$,
            then $G/G_p$ with its natural smooth structure is $G$-equivariantly
            diffeomorphic to $M$ {\cite[p.\ 1]{filipk}}.
            (See also {\cite[Thm.\ II.3.2, Prop.\ II.4.3(a)]{helgasonnew}},
            {\cite[Thm.\ II.1.1.2]{onishchik1}}, or {\cite[1.5.9]{wolf}}.)
        \item A group $G$ acts faithfully on the coset space $G/G_p$
            (i.e.\ is a subgroup of $\op{Diff}(G/G_p)$,
            as opposed to merely surjecting onto one)
            if and only if $G_p$ contains no nontrivial normal subgroups of $G$.
            (Proof: The kernel of $G \to \op{Diff}(G/G_p)$ is normal and
            contained in $G_p$ since it fixes the identity coset. Conversely,
            if $N \subseteq G_p$ is normal in $G$, then $G \to \op{Diff}(G/G_p)$
            factors through $G/N$.)
    \end{enumerate}

    Part (ii) is the following argument distilled from \cite[\S 3.3--3.4]{thurstonbook}.
    Fix any $G$-invariant Riemannian metric on $M = G/G_p$.
    Discrete subgroups $\Gamma$ correspond
    to orbifolds $N = \Gamma \backslash M$ covered by $M$.
    Such $N$ naturally inherits a Riemannian metric (i.e.\ is a manifold)
    if and only if $\Gamma$ acts freely on $M$,
    i.e.\ $\Gamma$ intersects no point stabilizer nontrivially.
    Since $G$ acts transitively on $M$ with compact point stabilizers,
    \begin{enumerate}
        \item every manifold quotient $\Gamma \backslash M$ is complete
            {\cite[Cor.\ 3.5.12 and Prop.\ 3.4.15]{thurstonbook}};
        \item every complete Riemannian manifold $N$ locally isometric to
            $M$ is isometric to a quotient space $\Gamma \backslash M$
            {\cite[Prop.\ 3.4.5 and 3.4.15]{thurstonbook}}; and
        \item $\Gamma \backslash M$ has finite volume if and only if
            $\Gamma \backslash G$ does (i.e.\ $\Gamma$ is a lattice).
                (Proof: Let $\omega$ be a left-invariant volume form on $G$.
                Cosets of $G_p$ are compact, so one may integrate along them, recovering
                the invariant volume form on $M$ up to a scale factor. Then
                \(\op{vol}(\Gamma \backslash M)\)
                is a multiple of \(\op{vol}(\Gamma \backslash G)\).)
    \end{enumerate}

    Part (iii) is merely the original definition
    (\ref{defn:geometries}(iii)) rephrased in terms of the
    correspondence from part (i).
\end{proof}

\section{Classification of isotropy representations (point stabilizers)}
\label{chap:isotropy}

This section presents the first step in the strategy of Thurston and Filipkiewicz,
which is to classify the linear isotropy representations---representations
of point stabilizers $G_p$ on tangent spaces $T_p M$.

Since an isometry of a connected Riemannian manifold
is determined by its value and derivative at a point
(see e.g.\ \cite[Prop.~A.2.1]{bp}), 
any such action is faithful and preserves the Riemannian metric on $T_p M$.
Since $G_p$ is connected
by a homotopy exact sequence calculation \cite[Prop.~1.1.1]{filipk}
and compact,
classifying representations $G_p \curvearrowright T_p M$
is equivalent to classifying closed connected subgroups $G_p \subseteq \op{SO}(5)$,
up to conjugacy in $\op{GL}(5,\R)$. The classification is as follows,
summarized in Figure \ref{fig:isotropy_poset}.\footnote{
    Proofs of the inclusions are omitted;
    all are readily guessed except for
    $S^1_{1/2} \subset \op{SO}(3)_5$.
    This last inclusion can be seen from the description of $\op{SO}(3)_5$
    as $\op{SO}(3)$ acting by conjugation on traceless symmetric
    bilinear forms---a $90$ degree rotation in the first
    two coordinates has order $4$ but acts with order $2$
    on the diagonal matrix $d(1,-1,0)$.
}

\begin{prop}[\textbf{Classification of isotropy representations}] \label{isotropy_classification}
    The closed connected subgroups of $\op{SO}(5)$ are,
    up to conjugacy in $\op{GL}(5, \R)$,
    \begin{itemize}
        \item one of the following groups, acting by its standard representation over $\R$
            on a subspace of $\R^5$ and trivially on the orthogonal complement;
            \begin{align*}
                &{}\{1\}  &
                &{}\op{SO}(2)  &
                &{}\op{SO}(3)  &
                &{}\op{SO}(4)  &
                &{}\op{SO}(5)  &
                &{}\op{SO}(2) \times \op{SO}(2)  &
                &{}\op{SO}(2) \times \op{SO}(3)
            \end{align*}
        \item one of $\op{SU}(2)$ or $\op{U}(2)$ acting by its standard representation
            on $\C^2 \cong \R^4 \subset \R^5$;
        \item $\op{SO}(3)_5$, a copy of $\op{SO}(3)$ acting irreducibly on $\R^5$; or
        \item $S^1_{m/n}$ $(0 \leq \frac{m}{n} \leq 1)$,
            the 1-parameter subgroup of $\op{SO}(2) \times \op{SO}(2)$ defined by
            \[ S^1 \ni z \mapsto (z^m, z^n) \in S^1 \times S^1 \cong \op{SO}(2) \times \op{SO}(2). \]
    \end{itemize}
\end{prop}

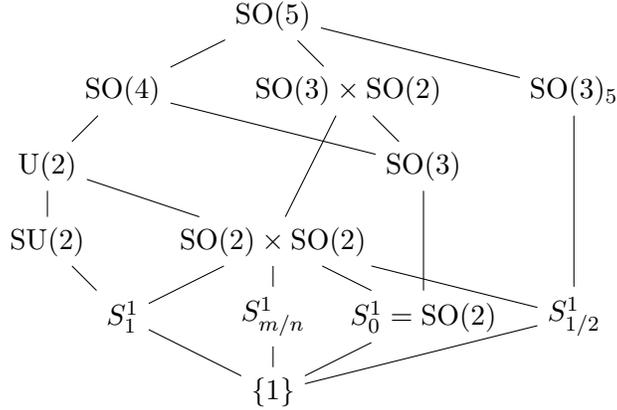
\begin{figure}[h!]
    \caption{Closed connected subgroups of $\op{SO}(5)$, with inclusions.}
    \label{fig:isotropy_poset}
    \label{fig_thesis:isotropy_poset}
    \begin{center}\begin{tikzpicture}
        \draw (0, 5) node (so5) {$\op{SO}(5)$};
        \draw (-2, 4) node (so4) {$\op{SO}(4)$};
        \draw (1, 4) node (so3so2) {$\op{SO}(3) \times \op{SO}(2)$};
        \draw (4, 4) node (so35) {$\op{SO}(3)_5$};

        \draw (-3, 3) node (u2) {$\op{U}(2)$};
        \draw (-3, 2) node (su2) {$\op{SU}(2)$};

        \draw (2, 3) node (so3) {$\op{SO}(3)$};
        \draw (0, 2) node (t2) {$\op{SO}(2) \times \op{SO}(2)$};

        \draw (-2, 1) node (s11) {$S^1_1$};
        \draw (0, 1) node (s1q) {$S^1_{m/n}$};
        \draw (2, 1) node (so2) {$S^1_0 = \op{SO}(2)$};
        \draw (4, 1) node (s12) {$S^1_{1/2}$};

        \draw (0, 0) node (triv) {$\{1\}$};

        \draw (so5) -- (so4) -- (u2) -- (su2) -- (s11) -- (triv);
        \draw (so5) -- (so3so2) -- (so3) -- (so2) -- (triv);
        \draw (so5) -- (so35) -- (s12) -- (triv);
        \draw (so4) -- (so3);
        \draw (so3so2) -- (t2) -- (so2);
        \draw (u2) -- (t2) -- (s11);
        \draw (t2) -- (s1q) -- (triv);
        \draw (t2) -- (s12);
    \end{tikzpicture}\end{center}
\end{figure}

\subsection{Proof of the classification of isotropy representations}
\label{sec:isotropy_proof}

The proof presented below is similar in spirit to that of \cite[\S 1.2]{filipk},
albeit with somewhat more representation-theoretic language.
In higher dimensions, it would be advantageous to use an alternative
strategy that provided more structure and fewer opportunities for human error.
One such approach, to list maximal subgroups recursively
using results by Dynkin, is outlined in
\cite[Tables 1, 3]{kerrkollross} for several groups including
$\op{SO}(5)$ and $\op{SO}(6)$.

In dimension $5$, the groups involved are small enough that only
representations of $S^1$ and $\op{SU}(2)$ really need to be understood;
this permits a slightly more elementary approach, using some standard
facts about representations of low-dimensional groups.
These will be stated where used, with proofs deferred
to Section \ref{sec:low_d_reprs}
so that they are available but not a source of clutter.

The first step is the classification up to finite covers
by using the classification of compact simple Lie groups.
\begin{prop} \label{covering}
    Any closed connected proper subgroup $K \subset \op{SO}(5)$ is finitely covered
    by a product of at most two factors, each of which is $S^1$ or $\op{SU}(2)$.
\end{prop}
\begin{proof}
    Every compact connected Lie group is finitely covered by a product
    of circles $S^1$ and compact simply-connected simple Lie groups
    \cite[Thm.~V.8.1]{brockerdieck}.
    By the classification of compact connected simple Lie groups
    \cite[Ch. X, \S{}6 (p. 516)]{helgasonnew},
    the simple factors could be
    \begin{itemize}
        \item $\op{Sp}(n)$, $n \geq 1$ (dim. $n(2n+1)$)
        \item $\op{SU}(n)$, $n \geq 3$ (dim. $n^2 - 1$)
        \item $\op{Spin}(n)$, $n \geq 7$ (dim. $\binom{n}{2}$)
        \item $G_2$ (dim. $14$), $F_4$ (dim. $52$), $E_6$ (dim $78$),
            $E_7$ (dim. $133$), or $E_8$ (dim. $248$).
    \end{itemize}
    The dimension of $\op{SO}(5)$ is $10$; so only $\op{Sp}(1) \cong \op{SU}(2)$
    and $\op{SU}(3)$ could be simple factors
    of a cover of a proper subgroup of $\op{SO}(5)$.
    Moreover, $\op{SU}(3)$ is not a factor since it has no
    nontrivial $5$-dimensional representations (Prop.~\ref{nosu3}).

    Hence $K$ is finitely covered by a product whose factors
    are $S^1$ or $\op{SU}(2)$.
    Since $\op{SO}(5)$ has rank $2$, such a product has at most two factors.
\end{proof}

At this stage, the full classification (Prop.~\ref{isotropy_classification})
reduces to listing low-dimensional representations of $S^1$ and $\op{SU}(2)$.
\begin{proof}[Proof of Prop.~\ref{isotropy_classification}
    (Classification of closed connected subgroups $K \subseteq \op{SO}(5)$)]
    By Prop.~\ref{covering} above,
    a closed connected subgroup $K \subseteq \op{SO}(5)$ either is $\op{SO}(5)$
    or has a finite cover $\tilde{K} \cong K_1 \times K_2$, where
    $K_1$ and $K_2$ are $S^1$ or $\op{SU}(2)$.
    When $\tilde{K}$ is a product, there are the following cases.

    \paragraph{Case 1: $\tilde{K} \cong \{1\}$.} Then $K = \{1\}$.
    \paragraph{Case 2: $\tilde{K} \cong S^1$.} Then $K$ is the image
            of some homomorphism  $f: \R \to \op{SO}(5)$. The eigenvalues of
            $f(t) \in \op{SO}(5)$ are $1$ and two pairs of complex
            conjugates of norm $1$, with homomorphic dependence on $t$;
            so there is some decomposition of $\R^5$
            as $\C \times \C \times \R$ in which
                \[ f(t)
                    = (e^{xt}, e^{yt}, 1)
                    \in \op{U}(1) \times \op{U}(1) \times \op{SO}(1) . \]
            Since the image is closed, $x$ and $y$ must be rationally
            dependent; so $K = S^1_{m/n}$ where either $\frac{x}{y}$
            or $\frac{y}{x}$ is equal to $\frac{m}{n}$.
    \paragraph{Case 3: $\tilde{K} \cong S^1 \times S^1$.}
            Since the two $S^1$
            factors commute, they share eigenspaces. Using
            $\R^5 \cong \C \times \C \times \R$ as above, $K$ must be
            a $2$-dimensional subgroup of
            $\op{U}(1) \times \op{U}(1) \times \op{SO}(1)
            \cong \op{SO}(2) \times \op{SO}(2)$.
    \paragraph{Case 4: $\tilde{K} \cong \op{SU}(2)$.}
            The classification
            of irreducible representations of $\op{SU}(2)$ (Prop.~\ref{reprs_su2})
            implies that $K$ is one of $\op{SU}(2)$ acting on $\C^2$,
            $\op{SO}(3)$ acting on $\R^3$, and $\op{SO}(3)$ acting irreducibly on $\R^5$.
    \paragraph{Case 5: $\tilde{K} \cong \op{SU}(2) \times S^1$.}
            This
            $\op{SU}(2)$ cannot act irreducibly on $\R^5$ since that
            action has trivial centralizer (Cor.~\ref{reprs_so3}).
            Hence the $\op{SU}(2)$ factor acts as either
            $\op{SO}(3)$ on $\R^3$ or $\op{SU}(2)$ on $\C^2$.
            \begin{itemize}
                \item If $\op{SU}(2)$ acts as $\op{SO}(3)$ on some $\R^3 \subset \R^5$,
                    then the $S^1$ factor preserves the decomposition
                    $\R^5 \cong_{\op{SO}(3)} \R^3 \oplus 2\R$ since it commutes with $\op{SO}(3)$.
                    So $K \subseteq \op{S}(\op{O}(3) \times \op{O}(2))$,
                    and by connectivity $K \subseteq \op{SO}(3) \times \op{SO}(2)$.
                    Equality holds since the dimensions match.
                \item If $\op{SU}(2)$ acts as $\op{SU}(2)$ on some $\C^2 \subset \R^5$,
                    then as above
                    $K \subseteq \op{SO}(4) \times \op{SO}(1) \cong \op{SO}(4)$.
                    Since $\op{SO}(4)$ is covered by $\op{SU}(2) \times \op{SU}(2)$,
                    the $S^1$ factor can be taken to lie in the second $\op{SU}(2)$
                    as a $1$-parameter subgroup---all of which are conjugate.
                    Thus there is only one action of
                    $\op{SU}(2) \times S^1$ on $\R^4$ up to conjugacy.
                    It can be realized as the action by $\op{SU}(2)$ on $\C^2$
                    along with the action of the unit-norm scalars,
                    which amounts to $\op{U}(2)$.
            \end{itemize}
    \paragraph{Case 6: $\tilde{K} \cong \op{SU}(2) \times \op{SU}(2)$.}
            As in Case 5 above, $K \subseteq \op{SO}(4)$. Since
            the dimensions match, $K = \op{SO}(4)$.
\end{proof}

\subsection{Standard facts about representations of lower-dimensional groups}
\label{sec:low_d_reprs}

This subsection collects some basic facts about representations of
$\op{SU}(2)$, $\op{SU}(3)$, and $\op{SO}(3)$.
These are only needed for the above classification of isotropy representations
(Prop.~\ref{isotropy_classification}),
so a reader familiar with these groups may wish to skip this subsection.

\begin{prop} \label{nosu3}
    Representations of $\op{SU}(3)$ of dimension less than $6$ over $\R$ are trivial.
\end{prop}
\begin{proof}
    The complex irreducible representations of
    $\mathfrak{su}_3 \C \cong \mathfrak{sl}_3 \C$
    \footnote{
        For the isomorphism, write $\mathfrak{sl}_n \R$
        as the sum of its skew-symmetric part $\mathfrak{k}$ and
        symmetric part $\mathfrak{p}$. Then since
        $\mathfrak{su}_n$ consists of the matrices
        $A$ which satisfy $A + A^* = 0$ and are traceless,
        $\mathfrak{su}_n = \mathfrak{k} + i\mathfrak{p}$.
    }
    exponentiate to those of the simply-connected group $\op{SU}(3)$.
    Listing highest weights as done in \cite[Ch.~12]{fultonharris}
    shows that they have dimensions $1$, $3$, $6$, and higher.
    Over $\R$, an irreducible representation over $\C$
    either remains irreducible or splits into two isomorphic irreducible summands
    \cite[Prop.~II.6.6(vii-ix)]{brockerdieck};
    so any real irreducible representation has
    dimension $1$, $2$, $3$, $6$, or higher.

    Then choosing an invariant inner product, an irreducible action of $\op{SU}(3)$
    on $\R^k$ ($k < 6$) factors through $\op{SO}(3)$.
    Since $\op{SU}(3)$ is simple and of higher dimension than
    $\op{SO}(3)$, its image in $\op{SO}(3)$ is trivial.
\end{proof}

\begin{prop} \label{reprs_su2}
    Let $V$ be the standard representation of $\op{SU}(2)$ over $\C$.
    The finite-dimensional irreducible representations of $\op{SU}(2)$ over $\R$ are
    \begin{enumerate}[(i)]
        \item for even $n \geq 0$, an invariant real subspace
            of $\op{Sym}^n V$ (dimension $n+1$); and
        \item for odd $n \geq 1$, the representation $\op{Sym}^n V$
            taken as a real vector space (dimension $2(n+1)$).
    \end{enumerate}
\end{prop}
\begin{proof}
    The irreducible representations of $\op{SU}(2)$ over $\C$
    are the symmetric powers of $V$ \cite[Prop.~II.5.1, II.5.3]{brockerdieck}
    (see also \cite[11.8]{fultonharris}).
    The action can be written explicitly as
        \[
            \begin{pmatrix} a & b \\ -\ol{b} & -\ol{a} \end{pmatrix}
            f(x,y)
            =
            f(ax + by, -\ol{b}x + \ol{a}y)
        \]
    where $f \in \op{Sym}^n V$ is homogeneous of degree $n$
    in $\C[x,y]$.

    Define $\phi: \op{Sym}^n V \to \op{Sym}^n V$ by
        \[ (\phi f)(x,y) = \ol{f(\ol{y}, -\ol{x})} . \]
    One can verify from this formula that $\phi$ is
    $\op{SU}(2)$-equivariant and satisfies
    $\phi i = -i \phi$ and $\phi^2 = (-1)^n$.
    \begin{enumerate}[(i)]
        \item If $n$ is even, then $\phi$ is a real structure on $\op{Sym}^n V$.
            An irreducible representation over $\C$ with a real structure
            is the direct sum of two copies of an irreducible representation over $\R$
            \cite[Prop.~II.6.6(vii)]{brockerdieck}.
            A summand can be recovered as
            $(\phi + \op{Id})(\op{Sym}^n V)$,
            consisting of the polynomials
            $f(x,y) = \sum_{m} f_m x^m y^{n-m}$
            where $f_m = (-1)^m \ol{f_{n-m}}$.
        \item If $n$ is odd, then $\phi$ is a quaternionic structure
            on $\op{Sym}^n V$. An irreducible representation over $\C$
            with a quaternionic structure is irreducible over $\R$
            \cite[Prop.~II.6.6(ix)]{brockerdieck}. \qedhere
    \end{enumerate}
\end{proof}
\begin{rmk}[\hspace{1sp}{\cite[II.5.4]{brockerdieck}}] \label{rmk:isotropy_so3reprs}
    If and only if $n$ is even,
    the scalar $-1 \in \op{SU}(2)$ acts trivially on $\op{Sym}^n V$,
    allowing the action of $\op{SU}(2)$ to descend to an action of
    $\op{SO}(3) \cong \op{SU}(2)/\{\pm 1\} = \op{SU}(2)/Z(\op{SU}(2))$.
\end{rmk}

\begin{prop}
    $\op{End} W \cong \R$ for any
    irreducible representation $W$ of $\op{SO}(3)$ over $\R$.
\end{prop}
\begin{proof}
    Every endomorphism of $W$ extends to an endomorphism of $W \otimes \C$,
    which is irreducible over $\C$.
    By Schur's lemma, $\phi \otimes \C$ is multiplication by a scalar.
    Since $\op{dim} W$ is odd (Rmk.~\ref{rmk:isotropy_so3reprs}),
    the only scalars that preserve $W$ are the reals.
\end{proof}
\begin{cor} \label{reprs_so3}
    The centralizer of $\op{SO}(3)_5$ in $\op{SO}(5)$ is trivial.
\end{cor}

\section{The case of irreducible isotropy}
\label{chap:irr}

When $G_p \curvearrowright T_p M$ is irreducible,
the classification of isotropy irreducible homogeneous spaces
and the classification of irreducible Riemannian symmetric spaces
produce a list (Prop.~\ref{thm:irr_list})
of homogeneous spaces as candidates.
Then it only remains to check that each is a maximal model
geometry with irreducible isotropy
(Prop.~\ref{prop:irr_model}--\ref{prop:irr_maximal}).
Taken together, these results prove Thm.~\ref{thm:main}(i).

\subsection{The list of candidates}

The first step of this classification is to obtain,
from the classification of isotropy irreducible homogeneous spaces,
the following explicit list of candidate geometries.
\begin{prop} \label{thm:irr_list}
    Let $M = G/G_p$ be a $5$-dimensional maximal geometry for which
    $G_p \curvearrowright T_p M$ is irreducible.
    Then $M$ is one of the following Riemannian symmetric spaces.
    \begin{align*}
        \Euc^5 &= \R^5 \rtimes \op{SO}(5) / \op{SO}(5)  \\
        S^5 &= \op{SO}(6)/\op{SO}(5)  &
        \op{SU}(3)&/\op{SO}(3)  \\
        \Hyp^5 &= \op{SO}(5,1)/\op{SO}(5)  &
        \op{SL}(3,\R)&/\op{SO}(3)
    \end{align*}
\end{prop}

Aside from consulting existing classifications,
only the following standard fact is needed to produce the above list.
\begin{lemma} \label{only_one_so3_in_su3}
    All copies of $\op{SO}(3)$ in $\op{SU}(3)$ are conjugate.
\end{lemma}
\begin{proof}
    The irreducible representations of $\op{SO}(3)$ over $\C$ of dimension
    $3$ or lower are the trivial and the standard representation;
    so any two embeddings $\op{SO}(3) \hookrightarrow \op{SU}(3)$
    are conjugate by some representation isomorphism $h \in \op{GL}(3,\C)$.

    Isomorphic unitary representations are unitarily isomorphic
    (see e.g.\ \cite[Exercise II.1.8]{brockerdieck},
    or a proof in \cite[Prop.~5.2.1]{barut1986} using polar decomposition).
    Then $h$ can be taken to lie in $\op{SU}(3)$,
    so the two embeddings of $\op{SO}(3)$ are conjujgate in $\op{SU}(3)$.
\end{proof}

\begin{proof}[Proof of Prop.~\ref{thm:irr_list}]
    Assume $M = G/G_p$ is a $5$-dimensional maximal geometry
    such that $G_p \curvearrowright T_p M$ is irreducible.
    It is a theorem of Wolf that if a homogeneous space $G/G_p$
    has compact, connected, irreducibly-acting isotropy $G_p$,
    then either $G/G_p$ is a Riemannian symmetric space
    or $G$ is a compact simple Lie group
    \cite[Thm.~1.1]{wolf_irr}.\footnote{
        In \cite[Table, p.\ 107--110]{wolf_irr},
        Wolf gave a more explicit classification of
        strongly isotropy irreducible spaces.
        Wang and Ziller remark in \cite[p.\ 2]{wangziller1991} that
        this classification has an omission but do not say whether the erratum
        in \cite{wolf_irr_fix} completes it. Instead they refer the reader to two other
        classifications, by Manturov in \cite{manturov1, manturov2, manturov3}
        (earlier, also with omissions; see also \cite{manturov1998})
        and by Kr\"amer in \cite{kramer1975} (believed complete).
        A slightly weaker version of the result used here---omitting
        the claim that $G$ is simple---was known to Matsushima,
        with proof first given by Nagano in \cite[Appendix]{nagano1959}.
    }
    From the classification of isotropy representations $G_p \curvearrowright T_p M$
    (Prop.~\ref{isotropy_classification}), $G_p$ is either $\op{SO}(5)$
    or $\op{SO}(3)_5$.
    \paragraph{Case 1: $G_p = \op{SO}(5)$.}
        Since $\op{SO}(5)$ acts transitively on $2$-planes through the origin in $\R^5$,
        $M$ has constant sectional curvature---and is therefore
        exactly one of $\Euc^5$, $S^5$, and $\Hyp^5$ by the
        Killing-Hopf theorem \cite[Cor.\ 2.4.10]{wolf}.
    \paragraph{Case 2: $G_p = \op{SO}(3)_5$.} In this case, $G$ is $8$-dimensional.
        If $M$ is an irreducible symmetric space, then $M$
        is Euclidean, $\op{SL}(3,\R)/\op{SO}(3)$, or $\op{SU}(3)/\op{SO}(3)$
        by the classification of irreducible symmetric spaces
        (see \cite[X.6 Table V and p.515--518]{helgasonnew}).

        Otherwise, $G$ is an $8$-dimensional
        compact simple Lie group. By the classification of compact
        simple Lie groups \cite[X.6 Table IV]{helgasonnew},
        $G \cong \op{SU}(3)$.
        Since all copies of $\op{SO}(3)$ in $\op{SU}(3)$ are conjugate
        (Lemma \ref{only_one_so3_in_su3}),
        $X$ is the symmetric space $\op{SU}(3)/\op{SO}(3)$.
\end{proof}

\begin{rmk}
    We have not explicitly verified that the candidate spaces have irreducible
    isotropy. That this holds for every irreducible
    Riemannian symmetric space
    (which seems to be well known; see
    e.g.\ \cite[Ch.~1, condition (v)]{wolf_irr})
    is a side effect of proving the usual decomposition theorem for
    Riemannian symmetric spaces by using orthogonal involutive Lie algebras.
    Such an approach can be found in \cite[Thm.~8.3.8]{wolf}.

    Alternatively, it follows in the case of $\op{SO}(5)$
    from $\op{SO}(5)$ having a transitive action on $S^4$;
    and in the case of $\op{SO}(3)_5$ from observing that
    $\op{SO}(3)$ acts irreducibly on the space $V$ of traceless
    symmetric $3 \times 3$ matrices, and writing
    \begin{align*}
        \lie{sl}_3 \R &= \lie{so}_3 \R + V  &
        \lie{su}_3 \R &= \lie{so}_3 \R + iV .
    \end{align*}
\end{rmk}

\subsection{Maximality and the existence of compact quotients}

It happens that
all non-Euclidean isotropy irreducible spaces---with two exceptions,
neither of which has dimension $5$---are
already known to be maximal model geometries.
So to prove that the $5$-dimensional isotropy irreducible geometries
(i.e.\ those produced in Prop.~\ref{thm:irr_list}) are maximal model geometries,
it suffices to collect some existing theorems.

\begin{prop} \label{prop:irr_model}
    Any geometry with irreducible isotropy is a model geometry
    (i.e.\ admits a compact manifold quotient).
\end{prop}
\begin{proof}
    As part of the classification of strongly isotropy irreducible spaces,
    such a geometry is either already compact or Riemannian symmetric
    \cite[Thm.~1.1]{wolf_irr}.
    Borel proved in \cite[Thm.~A]{borel1963compact}
    that every simply-connected Riemannian symmetric space
    $G/K$ admits a compact manifold quotient $\Gamma\backslash G/K$.
\end{proof}

Since all of the $5$-dimensional isotropy irreducible geometries
are Riemannian symmetric spaces, it suffices to know when a
Riemannian symmetric space is maximal.
The geometry $\Euc^5$ is maximal since its isotropy $\op{SO}(5)$ is maximal
(Rmk.~\ref{rmk:maximality_by_isotropy});
and the other candidates $G/K$ have semisimple $G$,
so the following result proves them maximal.

\begin{prop} \label{prop:irr_maximal}
    Suppose $G/K$ is a Riemannian symmetric space---i.e.\ suppose
    $K$ is an open subgroup of the fixed set in $G$ of some
    order $2$ element of $\op{Aut} G$.
    Suppose further that $G$ is semisimple and acts faithfully on $G/K$.
    Then $G/K$ is a maximal geometry.
\end{prop}
\begin{proof}
    If $G/K$ is a Riemannian symmetric space with $G$ semisimple and acting faithfully,
    then $G = (\op{Isom} G/K)^0$ in every $G$-invariant metric on $G/K$
    \cite[Thm.~V.4.1(i)]{helgasonnew}.
    At least one of these invariant metrics has an isometry group whose identity component
    is the transformation group $G'$ of a maximal geometry $G'/K'$ realizing $G/K$
    (Prop.~\ref{prop:maximality_with_metric}).
    Then $G' = G$, so $G/K$ is maximal.
\end{proof}

\begin{rmk}
    Wolf also proved maximality in the non-symmetric case in \cite[Thm.~17.1]{wolf_irr}.
    That is, except for
    $G_2/\op{SU}(3) \cong S^6$
    and $\op{Spin}(7)/G_2 \cong S^7$,
    a simply-connected
    isotropy irreducible Riemannian homogeneous space $G/K$ is maximal
    if $G$ is semisimple (i.e.\ $G/K$ is not Euclidean) and acts faithfully.
    The same theorem also includes a description of the full isometry group---not
    just the identity component.
\end{rmk}

\begin{rmk}
    One could instead verify maximality
    by checking that the listed spaces
    are not extended by any geometries with larger isotropy groups.
    The constant-curvature geometries have maximal isotropy;
    $\op{SL}(3,\R)/\op{SO}(3)$ is distinguished from the constant-curvature
    geometries by having rank $2$
    \cite[Appendix 5 \S{2}, p.~242]{ballmangromovschroeder};
    and $\op{SU}(3)/\op{SO}(3)$ is distinguished by having nonzero $\pi_2$,
    which can be calculated using the homotopy exact sequence.
\end{rmk}

\section{The case of trivial isotropy: solvable Lie groups}
\label{chap:sol}

This section proves Thm.~\ref{thm:main}(ii),
the classification of $5$-dimensional maximal model geometries $M = G/G_p$
for which $G_p \curvearrowright T_p M$ is trivial.
(The reader may wish to consult the identification key in
Figure \ref{fig:solvable_flowchart} for a reminder of the results.)

\paragraph{Overview (see also ``Roadmap'' below).}
Our strategy, following that of Filipkiewicz in \cite[\S 6]{filipk},
begins by invoking Filipkiewicz's reduction to a classification of
simply-connected solvable groups.
\begin{prop}[\hspace{1sp}{\cite[Prop.\ 6.1.3]{filipk}}] \label{prop:filipk_solvable_groups}
    If $M = G/G_p$ is a maximal model geometry with $0$-dimensional point stabilizers,
    then $M \cong G$ is a connected, simply-connected, unimodular solvable
    Lie group, and $\op{Aut}(G)$ is solvable and simply-connected.
\end{prop}
The classification proceeds by expressing $G$ as an extension
of an abelian group by a nilpotent group (such as the nilradical).
Conveniently, only split extensions are needed
in order to produce the maximal geometries.
That is, Section \ref{sec:semidirect_5d} will prove that
\begin{prop} \label{prop:semidirect_5d}
    If $G = G/\{1\}$ is a maximal model geometry of dimension $5$,
    then either
    \begin{enumerate}[(i)]
        \item $G \cong \R^3 \rtimes \R^2$ where $\R^2$ acts on $\R^3$
            as the diagonal matrices with positive entries and
            determinant $1$; or
        \item $G \cong N \rtimes \R$
            where $N$ is nilpotent, connected, and simply-connected.
    \end{enumerate}
\end{prop}
The problem then reduces to classifying semidirect products
and checking lattice existence (for model geometries) and maximality.
We perform this classification in the language of Lie algebras,
using the correspondence between Lie algebras
and connected, simply-connected Lie groups
(see e.g.\ \cite[Thm.~I.2.2.10--11]{onishchik1} and \cite[Thm.~1.4.2]{onishchik3}).

\paragraph{Roadmap.}
After Section \ref{sec:solvable_notations} lists some notation,
Section \ref{sec:semidirect_5d} proves the above proposition
using some Lie algebra cohomology.
Details on the $\R^3 \rtimes \R^2$ geometry (in (i))
are in Section \ref{sec:solvable_r2_by_r3},
including an application of Dirichlet's unit theorem
(Prop.~\ref{prop:commutingmatrices}).
For each of the three groups that can occur as $N$ in (ii), 
a subsection of Section \ref{sec:semidirect_classification}
lists semidirect products,
omitting any that are easily shown not to produce a maximal model geometry;
and questions of lattice existence determine the model geometries.
Section \ref{sec:solvable_maximality} proves maximality
(Prop.~\ref{prop:solvable_maximality})
using a general theorem by Gordon and Wilson,
and lists features distinguishing the geometries from each other
(Prop.~\ref{prop:solvable_distinctness}).
Taken together, these results prove Thm.~\ref{thm:main}(ii).

\begin{rmk}
The $5$-dimensional solvable Lie algebras over $\R$ having already been classified by
Mubarakzyanov in \cite{muba_solvable5}
using a largely similar approach,\footnote{The classification is complete
in dimensions $\leq 6$ and is known in limited cases for higher dimensions;
see \cite[Introduction]{snobl2012} for a survey.}
much of this case could be reduced to consulting a table such as \cite[Table II]{patera}.
It is instead presented explicitly here,
since the method stands on its own and is illustrative---exposing tools and
calculations that will be reused in classifying the fibered geometries
(particularly a partial classification of extensions of $\R^2$ by $\R^3$ in
Lemmas \ref{lemma:r2_actions} and \ref{lemma:solvable_r2_by_r3_cohomology}).
Only a fraction of the classification of solvable Lie algebras over $\R$
is duplicated, since only those which admit lattices
and are tangent to maximal geometries are of interest.
\end{rmk}

\subsection{Notations}
\label{sec:solvable_notations}

The strategy outlined above for classifying solvable Lie groups
involves nilpotent subalgebras of their Lie algebras.
The following two definitions will aid in naming such subalgebras.
\begin{defn} \label{defn:nildefault}
    $\lie{n}_k$ is the semidirect sum $\R^{k-1} \semisum \R$
    where some $x_k \in \R$ acts on $\R^{k-1}$
    in its standard basis $\{x_1, \ldots, x_{k-1}\}$
    by a single Jordan block with eigenvalue $0$.
\end{defn}
\begin{eg}
    $\lie{n}_3$ is the $3$-dimensional Heisenberg Lie algebra,
    and $\lie{n}_4$ is the unique $4$-dimensional indecomposable
    nilpotent Lie algebra (see \cite[Table I]{patera}
    or Prop.~\ref{prop:nil_4d} below).
\end{eg}
\begin{defn}[\textbf{Nilradical}, see e.g.\ {\cite[\S 2.5]{onishchik3}}]
    The \keyword{nilradical} $\nilrad(\lie{g})$ of a Lie algebra $\lie{g}$
    is the unique maximal nilpotent ideal of $\lie{g}$.
\end{defn}

Lie algebra extensions that split are semidirect sums $\lie{h} \semisum \lie{g}$,
which are classified by the action
$\lie{g} \to \op{der}(\lie{h})/\op{ad}(\lie{h}) = \op{out}(\lie{h})$.
As the actions encountered are usually traceless
due to unimodularity considerations (such as in Lemma \ref{lemma:tracelessly} below),
we also make the following definition.
\begin{defn}
    If $\lie{h}$ is a unimodular Lie algebra
    (so that $\op{ad} \lie{h}$ acts tracelessly),
    $\op{sout} \lie{h} \subseteq \op{out} \lie{h}$
    denotes the subalgebra consisting of traceless outer derivations.
\end{defn}

\subsection{Reduction to semidirect products}
\label{sec:semidirect_5d}

This section proves Prop.~\ref{prop:semidirect_5d},
which asserts that a maximal model geometry
$G = G/\{1\}$ of dimension $5$ is one of two forms of semidirect product.
Since such a claim bears a close resemblance to
\cite[Prop.\ 6.1.4]{filipk} from the $4$-dimensional case,
an $n$-dimensional generalization such as the following may be of interest.
\begin{prop}[$G$ is an extension of an abelian group by a nilpotent group]
    \label{prop:solvable_extension_problem}
    Suppose $\lie{g}$ is a unimodular
    solvable Lie algebra of dimension $n > 1$.
    \begin{enumerate}[(i)]
        \item If $\lie{g}$ is not nilpotent then $\lie{g}$ is an extension
                \[ 0 \to \nilrad \lie{g} \to \lie{g} \to \R^k \to 0 \]
            for some $0 < k < n$,
            and $\lie{g}$ acts tracelessly on $\nilrad \lie{g}$.
        \item If $\lie{g}$ is nilpotent, then it is a semidirect sum of
            a nilpotent ideal $\lie{n}$ and $\R$, with $\R$
            acting tracelessly.
    \end{enumerate}
\end{prop}
The following outline summarizes
the gap between the above extension problems
and the two semidirect products described in
Prop.~\ref{prop:semidirect_5d}.
\begin{proof}[Proof outline of Prop.~\ref{prop:semidirect_5d}]
    The Lie algebra $\lie{g}$ of a model geometry $G = G/\{1\}$
    is unimodular and solvable
    (Prop.~\ref{prop:filipk_solvable_groups}).
    The nilradical $\nilrad \lie{g}$ of a $5$-dimensional unimodular
    solvable Lie algebra $\lie{g}$ is either $\R^3$ or of dimension at least $4$
    (Prop.~\ref{prop:solvable_nilalternative}).
    The two cases of Prop.~\ref{prop:semidirect_5d} are proven as follows.
    \begin{enumerate}[(i)]
        \item If $\nilrad \lie{g} \cong \R^3$ and $G$ is a maximal geometry,
            then the extension in Prop.~\ref{prop:solvable_extension_problem}(i) above
            is split and $\R^2$ acts on $\R^3$ by traceless diagonal matrices
            (Prop.~\ref{prop:split_r3}).
        \item If $\nilrad \lie{g}$ has dimension at least $4$,
            then $\lie{g}$ is an extension of $\R$ by a nilpotent algebra
            by Prop.~\ref{prop:solvable_extension_problem} above;
            and the extension splits since any linear map from $\R$
            is a homomorphism.
    \end{enumerate}
    The requirement that geometries are simply-connected then allows
    these Lie algebra results to apply to the corresponding Lie groups.
\end{proof}
To complete the proof, the following subsections each prove one component---the
$n$-dimensional extension problem (Prop.~\ref{prop:solvable_extension_problem}),
the restriction on $\nilrad \lie{g}$ (Prop.~\ref{prop:solvable_nilalternative}),
and the case when $\nilrad \lie{g} = \R^3$ (Prop.~\ref{prop:split_r3}).

\subsubsection{The general extension problem}

This section proves Prop.~\ref{prop:solvable_extension_problem}---the
description of unimodular solvable Lie algebras as extensions of abelian
algebras by nilpotent algebras.
The claims about actions being traceless will be proven
using the following observation.
\begin{lemma} \label{lemma:tracelessly}
    If $\lie{a}$ is an ideal in a unimodular solvable Lie algebra $\lie{g}$
    such that $\lie{g}/\lie{a}$ is unimodular, then
    $\lie{g}$ acts tracelessly on $\lie{a}$.
\end{lemma}
\begin{proof}
    Suppose $g \in \lie{g}$.
    Since $\lie{g}$ is unimodular, $\op{tr} \op{ad}_{\lie{g}} g = 0$.
    Since $\lie{g}/\lie{a}$ is unimodular, $\op{tr} \op{ad}_{\lie{g}/\lie{a}} g = 0$.
    The conclusion follows from
        \[ \op{tr} \op{ad}_{\lie{g}}
        = \op{tr} \op{ad}_{\lie{g}/\lie{a}} + \op{tr} \op{ad}_{\lie{a}} . \qedhere \]
\end{proof}

Armed with this,
let $\lie{g}$ be a unimodular solvable Lie algebra that we hope to write as an extension;
the proof of Prop.~\ref{prop:solvable_extension_problem}
divides into the following two almost-independent cases.

\begin{proof}[Proof of Prop.~\ref{prop:solvable_extension_problem}(i)
        (the non-nilpotent case)]
    The derived algebra of a finite-dimensional solvable Lie algebra
    is contained in the nilradical
    (see e.g.\ Chevalley's theorem, \cite[II.7 Thm.~13]{jacobson}),\footnote{
        Alternatively, one could use Lie's theorem that $\lie{g}$ has a faithful
        representation as upper-triangular matrices.
    }
    so $\lie{g}/\nilrad(\lie{g})$ is abelian, and thus unimodular.
    Then $\lie{g}/\nilrad(\lie{g})$
    is some $\R^k$ acting tracelessly on $\nilrad \lie{g}$
    (Lemma \ref{lemma:tracelessly}).
    Since $k = 0$ or $k = n$ would make $\lie{g}$ nilpotent,
    $0 < k < n$.
\end{proof}

\begin{proof}[Proof of Prop.~\ref{prop:solvable_extension_problem}(ii)
        (the nilpotent case)]
    Since $\lie{g}$ is nilpotent, $[\lie{g},\lie{g}] \neq \lie{g}$.
    Then any proper vector subspace $\lie{n}$ of $\lie{g}$
    containing $[\lie{g}, \lie{g}]$ is a nilpotent ideal.
    Taking $\lie{n}$ to be of codimension $1$ makes $\lie{g}/\lie{n} \cong \R$,
    so $\lie{g}$ is an extension
        \[ 0 \to \lie{n} \to \lie{g} \to \R \to 0 . \]
    As an extension of $\R$, this splits (any section of $\lie{g} \to \R$
    as a linear map is immediately a homomorphism);
    so $\lie{g} \cong \lie{n} \semisum \R$.
    This action is traceless by Lemma \ref{lemma:tracelessly}.
\end{proof}

\subsubsection{Nilradicals}

This section proves the following restriction on nilradicals of $\lie{g}$.
\begin{prop} \label{prop:solvable_nilalternative}
    The nilradical of a $5$-dimensional unimodular solvable Lie algebra
    is either $\R^3$ or of dimension at least $4$.
\end{prop}
Nilradicals of dimension $4$ are classified later, in Prop.~\ref{prop:nil_4d}.
The upcoming proof makes use of one technical lemma---that the action
of $\lie{g}/\nilrad(\lie{g})$ on $\nilrad(\lie{g})$ is somehow morally
as good as faithful.
\begin{lemma} \label{lemma:abelianembedding}
    Suppose $\lie{g}$ is a finite-dimensional solvable Lie algebra
    with $\lie{g}/\nilrad(\lie{g}) \cong \R^k$.
    If $f: \R^k \to \lie{g}$
    is any section of the quotient map as a map of vector spaces,
    then
        \[ \op{ad}|_{\nilrad(\lie{g})} \circ f: \R^k \to \op{der}(\nilrad(\lie{g})) \]
    is injective.
\end{lemma}
\begin{proof}
    Let $\lie{k}$ be the kernel of this map, and
    let $\lie{h}$ be the vector subspace
    $\nilrad(\lie{g}) + f(\lie{k})$ of $\lie{g}$.
    Since $\lie{g}/\nilrad(\lie{g})$ is abelian,
        \[ [\lie{g}, \lie{g}] \subseteq \nilrad \lie{g} \subseteq \lie{h} ;  \]
    so $\lie{h}$ is an ideal.
    This inclusion also implies
    $[\lie{h}, \lie{h}] \subseteq \nilrad \lie{g}$,
    which is the base case for
    the following induction that shows
    $\lie{h}$ is nilpotent.
    \begin{align*}
        \lie{h}^{i+1} = [\lie{h}, \lie{h}^i]
            &\subseteq [\nilrad(\lie{g}) + f(\lie{k}), \nilrad(\lie{g})^{i-1}]
            \subseteq \nilrad(\lie{g})^i + 0
    \end{align*}
    Then $\lie{h} = \nilrad \lie{g}$ by the definition of
    the nilradical, so $\lie{k} = 0$.
\end{proof}

Lemma \ref{lemma:abelianembedding} will also be used in a later section
for the classification in the case $\nilrad \lie{g} = \R^3$.
For now, its role is to provide constraints on dimension in
the proof of Prop.~\ref{prop:solvable_nilalternative}.

\begin{proof}[Proof of Prop.~\ref{prop:solvable_nilalternative}]
    The two nilpotent Lie algebras of dimension $3$ are $\R^3$ and $\lie{n}_3$
    (see e.g.\ \cite[Lec.~10]{fultonharris}, \cite[Table I]{patera},
    or \cite[Table 21.3]{maccallum});
    so it will suffice to show that $\nilrad \lie{g}$
    has dimension at least $3$ and is not $\lie{n}_3$.

    \paragraph{Step 1: $\nilrad \lie{g}$ has dimension at least $3$.}
    This follows from a bound by Mubarakzyanov
    on the dimension of the nilradical (see e.g.\ \cite[Thm.\ 2.5.2]{onishchik3})
    but can also be proven directly, as follows.

    First, $\nilrad(\lie{g}) \neq 0$ since that would imply
    $\lie{g} = \lie{g}/\nilrad(\lie{g}) = \R^5$, which has nilradical $\R^5$.
    Also, $\nilrad \lie{g} \ncong \R$ since Lemma \ref{lemma:abelianembedding}
    would then require some injective linear map
        \[ \lie{g}/\nilrad(\lie{g}) \cong \R^4 \to
        \op{der} \R \cong \lie{gl}_1 \R \cong \R . \]

    Similarly, if $\dim \nilrad \lie{g} = 2$, then any linear section
    of $\lie{g} \to \lie{g}/\nilrad(\lie{g}) \cong \R^3$ would induce
    an injective linear map $\R^3 \to \lie{gl}_2 \R$.
    In fact this map would have to land in $\lie{sl}_2 \R$ since
    unimodularity of $\lie{g}$ and $\lie{g}/\nilrad(\lie{g}) \cong \R^k$
    requires $\lie{g}$ to act tracelessly (Lemma \ref{lemma:tracelessly}).
    Then $\lie{sl}_2 \R$ would occur as a subalgebra
    of $\lie{g}$---in which case $\lie{g}$ would not be solvable
    since every term of its derived series would contain $\lie{sl}_2 \R$.
    Hence $\dim \nilrad \lie{g} \neq 2$.

    \paragraph{Step 2: The traceless outer derivation algebra $\op{sout}(\lie{n}_3)$
        is the isomorphic image of
        some $\lie{sl}_2 \R \subset \op{der}(\lie{n}_3)$.}
    Let $\lie{n}_3$ have basis $x_1, x_2, x_3$
    where $x_1$ is central and $[x_3,x_2] = x_1$.

    A derivation $D: \lie{n}_3 \to \lie{n}_3$ induces a linear map
    $\R x_2 + \R x_3 \to \lie{n}_3 / (\R x_2 + \R x_3) \cong \R x_1$.
    These are in bijection with the inner derivations, so up to subtracting an inner
    derivation $D(\R x_2 + \R x_3) \subset \R x_2 + \R x_3$. Then
        \[ Dx_1 = D[x_3,x_2] = [x_3, Dx_2] + [Dx_3,x_2], \]
    so relative to the basis $\{x_1, x_2, x_3\}$, the matrix of $D$ is
    \[ \begin{pmatrix}
        a+d &   & \\
            & a & b \\
            & c & d
    \end{pmatrix} . \]
    If $D$ is traceless, then $a + d = 0$.
    Then $\op{sout}(\lie{n}_3) \cong \lie{sl}_2 \R$,
    and a section of $\op{sder}(\lie{n}_3) \to \op{sout}(\lie{n}_3)$
    is given by the above matrix.

    \paragraph{Step 3: $\lie{n}_3$ is not $\nilrad(\lie{g})$.}
    An extension
        \[ 0 \to \lie{n}_3 \to \lie{g} \to \R^2 \to 0 \]
    defines a map $\R^2 \to \op{sout}(\lie{n}_3)$ by lifting to $\lie{g}$
    and taking brackets. Composition with the section from Step 2
    produces a homomorphism $\phi: \R^2 \to \lie{sl}_2 \R$.
    Since $\lie{sl}_2 \R$ admits no $2$-dimensional abelian subalgebras
    (any such would make $[\cdot,\cdot]: \Lambda^2 \lie{sl}_2 \R \to \lie{sl}_2 \R$
    fail to be surjective),
    $\phi$ has nonzero kernel.
    Then $\lie{n}_3$ cannot be the nilradical of $\lie{g}$---since if it were,
    Lemma \ref{lemma:abelianembedding} would require $\phi$ to be injective.
\end{proof}

\subsubsection{The elimination of non-split extensions}

The last main ingredient in Prop.~\ref{prop:semidirect_5d}
is the elimination of non-split extensions of $\R^2$ by $\R^3$
and extensions with actions other than the one specified, as follows.
\begin{prop} \label{prop:split_r3}
    Suppose $G = G/\{1\}$ is a maximal model geometry of dimension $5$
    and $\lie{g}$ is its Lie algebra. If $\nilrad \lie{g} \cong \R^3$,
    then $\lie{g}$ is the semidirect sum $\R^3 \semisum \R^2$
    where $\R^2$ acts by traceless diagonal matrices.
\end{prop}

The proof makes use of two main computations:
the classification of faithful actions $\R^2 \hookrightarrow \lie{sl}_3 \R$
(Lemma \ref{lemma:r2_actions})
and the classification of extensions using Lie algebra cohomology
(Lemma \ref{lemma:solvable_r2_by_r3_cohomology}).
These are carried out below, followed by the proof of Prop.~\ref{prop:split_r3}.

\begin{lemma} \label{lemma:r2_actions}
    Up to linear changes of coordinates in $\R^3$ and $\R^2$,
    there are six embeddings $\phi: \R^2 \to \lie{sl}_3 \R$;
    each sends $(x,y) \in \R^2$ to one of the following matrices.
    Blank entries are zero.
    \begin{align*}
        &\begin{pmatrix}
            0 & x & y \\
              & 0 & x \\
              &   & 0
        \end{pmatrix}
        &
        &\begin{pmatrix}
            0 &   & y \\
              & 0 & x \\
              &   & 0
        \end{pmatrix}
        &
        &\begin{pmatrix}
            0 & x & y \\
              & 0 & \\
              &   & 0
        \end{pmatrix}
        &
        &\begin{pmatrix}
            x & y & \\
              & x & \\
              &   & -2x
        \end{pmatrix}
        &
        &\begin{pmatrix}
            x & y & \\
            -y & x & \\
              &   & -2x
        \end{pmatrix}
        &
        &\begin{pmatrix}
            x &   & \\
              & y & \\
              &   & -x-y
        \end{pmatrix}
    \end{align*}
\end{lemma}
\begin{proof}
    Name the above embeddings $\phi_1$ through $\phi_6$,
    and let $\{e_1, e_2\}$ be a basis for $\R^2$.
    Suppose $\phi: \R^2 \to \lie{sl}_3 \R$ is an embedding.
    The strategy is to find the Jordan form for $\phi(e_1)$
    and determine what matrices commute with it in $\lie{sl}_3 \R$.

    \paragraph{Case 1: $\phi(\R^2)$ contains no matrices with nonzero eigenvalues.}
    Changing coordinates to put $\phi(e_1)$ in Jordan form
    (or something like it), and then computing centralizers, either
    \begin{align*}
        &\left\{\begin{array}{rcl}
        \vspace{1em}
        \phi(e_1) & = & \begin{pmatrix}
            0 & 1 & \\
              & 0 & 1 \\
              &   & 0
        \end{pmatrix} \\
        \phi(e_2) & = & \begin{pmatrix}
            a & b & c \\
              & a & b \\
              &   & a
        \end{pmatrix}\end{array}\right\}
        & &\text{or} &
        &\left\{\begin{array}{rcl}
        \vspace{1em}
        \phi(e_1) & = & \begin{pmatrix}
            0 &   & 1 \\
              & 0 & \\
              &   & 0
        \end{pmatrix} \\
        \phi(e_2) & = & \begin{pmatrix}
            a & b & c \\
              & e & d \\
              &   & a
        \end{pmatrix}\end{array}\right\}.
    \end{align*}
    Since the image of $\phi$ consists of traceless matrices
    with no nonzero eigenvalues, either
    \begin{align*}
        \phi &= \phi_1
        & &\text{ or } &
        \phi(e_2) &= \begin{pmatrix} 0 & b & c \\ & 0 & d \\ & & 0 \end{pmatrix} .
    \end{align*}
    In the latter case, by replacing $e_2$ with $e_2 - ce_1$ we may assume
    $c = 0$. Then if both $b$ and $d$ are nonzero, we may rescale coordinates
    in $\R^3$ to obtain $\phi_1(e_1)$ (so $\phi$ is conjugate to $\phi_1$).
    Otherwise, $\phi$ is conjugate to $\phi_2$ or $\phi_3$.

    The embeddings $\phi_1$, $\phi_2$, and $\phi_3$ are distinct,
    distinguished by the ranks and nullities
    of the matrices in their images.

    \paragraph{Case 2: $\phi(\R^2)$ contains a matrix with a nonzero eigenvalue.}
    Change coordinates in $\R^2$ to assume $\phi(e_1)$ has $1$ as
    an eigenvalue. Then it must have one of the following Jordan forms,
    for some $y \in \R$.
    \begin{align*}
        &\begin{pmatrix}
            1 & y & \\
              & 1 & \\
              &   & -2
        \end{pmatrix}
        &
        &\begin{pmatrix}
            1 & y & \\
            -y & 1 & \\
              &   & -2
        \end{pmatrix}
        &
        &\begin{pmatrix}
            1 &   & \\
              & y & \\
              &   & -1-y
        \end{pmatrix}
    \end{align*}
    Computing centralizers shows that $\phi$ is conjugate to
    $\phi_4$, $\phi_5$, or $\phi_6$, respectively.
    (If $\phi(e_1)$ is diagonal with two identical diagonal entries,
    then some element of its centralizer isn't diagonal
    with two identical diagonal entries but still
    has one of the above forms.)

    To show that none of these are conjugate to each other,
    observe that $\phi_4$, $\phi_5$, and $\phi_6$ each
    send $(2,1) \in \R^2$ to a matrix in Jordan form which
    is not the Jordan form of a matrix in the image
    of any other $\phi_i$.
\end{proof}

The following definition is needed for
Lemma \ref{lemma:solvable_r2_by_r3_cohomology}'s computation of cohomology;
a survey can be found in \cite{wagemann}
or \cite[\S 2--4]{alekseevsky_nonsuper}\footnote{
    An almost identical version, generalized to super (i.e.\ $\Z/2\Z$-graded)
    Lie algebras, has been published as \cite{alekseevsky_super}.
}.
\begin{defn}[\textbf{Lie algebra cohomology}, following {\cite[\S 2]{wagemann}}]
    \label{defn:liecoho}
    Let $M$ be a module of a Lie algebra $\lie{g}$ over a field $k$.
    The \keyword{Chevalley-Eilenberg complex} is the cochain complex
	is the cochain complex
		\[ C^p(\lie{g}, M) = \op{Hom}_k(\Lambda^p \lie{g}, M) \]
	with boundary maps
	\begin{align*}
		d_p: C^p &\to C^{p+1} \\
		(d_p c)(x_1, \ldots, x_{p+1})
			&= \sum_{1 \leq i < j \leq p+1} (-1)^{i+j}
            c\left( [x_i, x_j], x_1, \ldots, \hat{x_i}, \ldots, \hat{x_j}, \ldots, x_{p+1} \right) \\
				&\quad + \sum_{1 \leq i \leq p+1} (-1)^{i+1}
                    x_i c\left( x_1, \ldots, \hat{x_i}, \ldots, x_{p+1} \right)
	\end{align*}
    where $\hat{x_i}$ means $x_i$ should be omitted.
    The cohomology of $\lie{g}$ with coefficients in $M$
    is defined to be the cohomology of this complex
    and denoted $H^p(\lie{g}, M)$.
\end{defn}

\begin{lemma} \label{lemma:solvable_r2_by_r3_cohomology}
    If $\R^2$ acts on $\R^3$ via $\phi: \R^2 \to \lie{sl}_3 \R$,
    then $H^2(\R^2; \R^3) \cong \R^3/\phi(\R^2)(\R^3)$.
\end{lemma}
\begin{proof}
    All $3$-cochains are zero since $\Lambda^3 \R^2 = 0$,
    so every $2$-cochain is a cocycle.
    A $2$-cochain is a map $\Lambda^2 \R^2 \to \R^3$,
    which can be recovered from the image
    of a spanning element $e_1 \wedge e_2 \in \Lambda^2 \R^2$;
    this identifies the $2$-cocycles with $\R^3$.

    It then suffices to identify the $2$-coboundaries with $\phi(\R^2)(\R^3)$.
    First, for all $1$-cochains $c$,
        \[ (dc)(x,y) = \phi(x)c(y) - \phi(y)c(x) \in \phi(\R^2)(\R^3) . \]
    Conversely, given $u \in \R^2$ and $w \in \R^3$,
    take $u' \in \R^2$ linearly independent from $u$
    and define a $1$-cochain $c$ by $c(u) = 0$ and $c(u') = w$.
    Then
        \[ (dc)(u, u') = \phi(u) c(u') = \phi(u)(w) . \qedhere \]
\end{proof}

The proof of Prop.~\ref{prop:split_r3}---that only one extension of
$\R^2$ by $\R^3$ produces a maximal geometry with trivial isotropy and
nilradical $\R^3$---is within reach, now that the above data can be
used to describe the extensions in just enough detail to rule most
of them out.

\begin{proof}[Proof of Prop.~\ref{prop:split_r3}]
    Suppose $G = G/\{1\}$ is a maximal model geometry of dimension $5$
    and $\lie{g}$ is its Lie algebra, with $\nilrad \lie{g} \cong \R^3$.
    We have already established that $\lie{g}$ is an extension
        \[ 0 \to \R^3 \to \lie{g} \to \R^2 \to 0 \]
    where $\R^2$ acts faithfully and tracelessly
    by lifting to $\lie{g}$ and taking brackets
    (Prop.~\ref{prop:solvable_extension_problem} and Lemma \ref{lemma:abelianembedding}).

    \paragraph{Step 1: Use $H^2(\R^2; \R^3)$ to classify extensions.}
    For each of the actions $\phi: \R^2 \hookrightarrow \lie{sl}_3 \R$
    (Lemma \ref{lemma:r2_actions}), we have by
    Lemma \ref{lemma:solvable_r2_by_r3_cohomology} that
        \[ H^2(\R^2; \R^3) \cong \R^3 / \phi(\R^2)(\R^3) . \]
    Explicitly, if $\phi_1, \ldots, \phi_6$ name the actions
    in Lemma \ref{lemma:r2_actions} and $\{e_1, e_2, e_3\}$
    is the standard basis of $\R^3$, then
    \begin{align*}
        H^2(\R^2; \R^3) &\cong \begin{cases}
                \R e_3 &\text{ if } \phi = \phi_1 \text{ or } \phi_2 \\
                \R e_2 + \R e_3 &\text{ if } \phi = \phi_3 \\
                0 &\text{ if } \phi = \phi_4, \phi_5, \text{ or } \phi_6 .
            \end{cases}
    \end{align*}
    Isomorphism classes of extensions of $\R^2$ by $\R^3$
    are in bijection with classes $[c] \in H^2(\R^2; \R^3)$
    \cite[Cor.~9]{alekseevsky_nonsuper}; so the extension
    with $\phi_6$ splits.

    Moreover, defining relations for $\lie{g}$ can be recovered
    by using the cocycle as an $\R^3$-valued bracket on $\R^2$
    \cite[Eqn.\ 5.5]{alekseevsky_nonsuper}.
    Explicitly, identify $\lie{g}$ as vector space with $\R^3 \oplus \R^2$;
    and for $k_i \in \R^3$ and $q_i \in \R^2$, define
        \[ [(k_1, q_1), (k_2, q_2)] = ( \phi(q_1)(k_2) - \phi(q_2)(k_1), 0 ) . \]

    \paragraph{Step 2: $\phi_1, \ldots, \phi_4$ produce the wrong nilradical.}
    Let $v = (0,1) \in \R^2$, and let $\lie{n} = \R^3 + \R v$.
    Then $\lie{n}$ is an ideal
    since $\lie{g}^2 \subseteq \R^3 \subseteq \lie{n}$.

    Since $v$ acts
    by a nilpotent matrix on $\R^3$, using the distributive
    law to compute $\lie{n}^4$ yields
        \[ \lie{n}^4 = [v, [v, [v, \R^3]]] = 0. \]
    Thus $\lie{n}$ is also nilpotent;
    so $\phi_1, \ldots, \phi_4$ do not produce $\lie{g}$
    where $\nilrad \lie{g} \cong \R^3$.

    \paragraph{Step 3: $\phi_5$ produces a non-maximal geometry.}
    If $\phi = \phi_5$, then $G \cong (\C \times \R) \rtimes \R^2$,
    where $(x,y) \in \R^2$ acts as scaling by $e^{x+iy}$ on $\C$
    and by $e^{-2x}$ on $\R$.
    The action of $S^1 \subset \C$ on $\C$ commutes with this,
    so $S^1 \subseteq \op{Aut}(G)$.
    Then $G/{1} \cong G \rtimes S^1 / S^1$.
\end{proof}

The case $\phi = \phi_6$ has $\R^2$ acting on $\R^3$ by traceless diagonal
matrices; the following section will show that it produces a model geometry.

\subsection{The \texorpdfstring{$\R^3 \rtimes \R^2$}{R3 \& R2} geometry}
\label{sec:solvable_r2_by_r3}

When point stabilizers are $0$-dimensional,
a geometry admitting a finite-volume quotient by isometries
is a Lie group admitting a lattice (Prop.~\ref{prop:geometries3}).
So to show this geometry is a model geometry, it suffices to prove the following.

\begin{prop} \label{prop:commutingmatrices}
    The Lie group $\R^3 \rtimes \R^2$, where $\R^2$
    acts on $\R^3$ as the diagonal matrices with determinant 1
    and positive eigenvalues, admits a lattice.
\end{prop}
\begin{proof}
    By applying linear changes of coordinates in $\R^3$ and $\R^2$,
    it suffices to construct a group isomorphic to $\Z^3 \rtimes \Z^2$
    where the action is by diagonalizable matrices with positive eigenvalues.
    One can use the ring of integers of a number field, acted on
    by its group of units.

    If $K$ is a cubic number field, then its ring of integers $\mathcal{O}_K$
    is isomorphic as a group to $\Z^3$; and if $K$ is totally real,
    Dirichlet's unit theorem (see e.g.\ \cite[Thm.\ 7.4]{neukirch})
    implies that the group of units $\mathcal{O}_K^\times$ has rank
    $3 - 1 = 2$. So one may take $\mathcal{O}_K \rtimes 2U$ where
    $U \subseteq \mathcal{O}_K^\times$ is free abelian of rank $2$,
    and $2U$ consists of the squares in $U$
    so that the action will have positive eigenvalues.

    To obtain such a field, take $K = \Q[x]/(p(x))$
    where $p \in \Z[x]$ is a monic irreducible cubic
    with three distinct real roots.
\end{proof}
\begin{eg}
    Let $p(x) = 1 - 3x + x^3$, so $K \cong \Q[\alpha]$
    where $\alpha$ is any root of $p$. Then
    \begin{align*}
        \alpha(3 - \alpha^2) &= 3\alpha - \alpha^3 = 1 \\
        (1-\alpha)(2 - \alpha - \alpha^2) &= 2 - 3\alpha + \alpha^3 = 1 ,
    \end{align*}
    so $\alpha$ and $1 - \alpha$ are units in $\Z[\alpha] \subseteq \mathcal{O}_K$.

    To prove that $\alpha$ and $1-\alpha$ are independent in $\mathcal{O}_K^\times$,
    let $v_1,v_2,v_3 \in K \otimes_\Q \R$ be a basis of eigenvectors for
    multiplication by $\alpha$ on $K$ as a $\Q$-vector space. Let $U \subseteq \mathcal{O}_K^\times$
    be the subgroup generated by $\alpha$ and $1-\alpha$, and define a homomorphism
    \begin{align*}
        \phi: U &\to \R^3 \\
        a &\mapsto \big(\log |\lambda_1|, \log |\lambda_2|, \log |\lambda_3|\big)
            \text{ where } av_i = \lambda_i v_i .
    \end{align*}
    Then $\phi(\alpha) \approx (0.6, -1, 0.4)$ and $\phi(1-\alpha) \approx (1, -0.4, -0.6)$
    are linearly independent since they lie in non-opposite octants;
    so $U \cong \phi(U) \cong \Z^2$.

    The matrices by which $\Z^2$ acts on $\Z^3$
    are products of even powers of $\alpha$ and $1-\alpha$,
    expressed in the basis $(1, \alpha, \alpha^2)$. That is,
    the action $\Z^2 \to \op{Aut} \Z^3 = \op{SL}(3,\Z)$ is given by
    \begin{align*}
        x,y &\mapsto
            \begin{pmatrix}
                0 & 0 & -1 \\ 1 & 0 & 3 \\ 0 & 1 & 0
            \end{pmatrix}^{2x}
            \begin{pmatrix}
                1 & 0 & 1 \\ -1 & 1 & -3 \\ 0 & -1 & 1
            \end{pmatrix}^{2y} .
    \end{align*}
\end{eg}

\subsection{Classification of semidirect sums with \texorpdfstring{$\R$}{R}}
\label{sec:semidirect_classification}

Prop.~\ref{prop:semidirect_5d} has reduced the discovery of candidates
to a classification of semidirect products---that is,
classifying semidirect sums $\lie{n} \semisum \R$ where $\lie{n}$ is nilpotent
of dimension $4$.
The program followed here is to classify such $\lie{n}$,
determine possible actions of $\R$ on each $\lie{n}$,
and determine which semidirect sums produce model geometries
(i.e.\ are tangent to Lie groups admitting lattices).

\begin{prop}[see also {\cite[Table I]{patera}}] \label{prop:nil_4d}
    Any nilpotent Lie algebra $\lie{n}$ of dimension $4$
    is isomorphic to $\R^4$ or $\R \oplus \lie{n}_3$ or $\lie{n}_4$.
\end{prop}
\begin{proof}
    The dimension of $\lie{n}^2 = [\lie{n}, \lie{n}]$
    distinguishes the three Lie algebras above,
    so use this dimension to determine $\lie{n}$.
    \begin{itemize}
        \item $\lie{n}^2 \neq \lie{n}$ since $\lie{n}$ is nilpotent.
        \item If $\dim \lie{n}^2 = 3$, then pick
            $x \in \lie{n} \smallsetminus \lie{n}^2$. Then
                \[ \lie{n}^2
                    = (\R x + \lie{n}^2)^2
                    = [\R x, \lie{n}^2] + [\lie{n}^2, \lie{n}^2]
                    \subseteq \lie{n}^3 , \]
            so the lower central series stabilizes at $\lie{n}^2 \neq 0$,
            so this never occurs for nilpotent $\lie{n}$.
        \item If $\dim \lie{n}^2 = 2$, then $\dim \lie{n}^3$ is either $0$
            or $1$.
            \begin{itemize}
                \item If $\dim \lie{n}^3 = 0$, then $\lie{n}^2$ is central.
                    Pick nonzero:
                    \begin{align*}
                        x_3 &\notin \lie{n}^2 \\
                        x_2 &\notin \lie{n}^2 + \R x_3 \\
                        x_1 &\in \lie{n}^2 \text{ such that }
                            [x_3, x_2] = 0 \text{ or } x_1 \\
                        y &\in \lie{n}^2 \text{ completing these to a basis}.
                    \end{align*}
                    Then $\lie{n}$ is either $\R^4$ or $\R \oplus \lie{n}_3$.
                    (In fact, for both of these, $\lie{n}^2$ is too small.)
                \item If $\dim \lie{n}^3 = 1$, then choose nonzero:
                    \begin{align*}
                        x_1 &\in \lie{n}^3 \\
                        x_2 &\in \lie{n}^2 \smallsetminus \lie{n}^3 \\
                        x_4 &\in \lie{n} \smallsetminus \lie{n}^2
                            \text{ such that } [x_4, x_2] = x_1 \\
                        x_3 &\in \lie{n} \smallsetminus \lie{n}^2
                            \text{ such that } [x_4, x_3] = x_2 .
                    \end{align*}
                    Since $[x_3, x_2] \in \lie{n}^3$, it could be
                    any multiple of $x_1$. Replacing $x_3$ by an element
                    of $x_3 + \R x_4$ to make its bracket with $x_2$ zero
                    yields a basis demonstrating $\lie{n} \cong \lie{n}_4$.
            \end{itemize}
        \item If $\dim \lie{n}^2 = 1$, then $\lie{n}^2$ is central,
            as the last term in the lower central series.
            For any linear complement $V$ of $\lie{n}^2$,
            the Lie bracket induces $V \wedge V \to \lie{n}^2 \cong \R$,
            which is necessarily degenerate since $V$ has odd dimension.
            Then we can choose $x_1$ spanning $\lie{n}^2$,
            $x_2$ and $x_3$ such that $[x_3, x_2] = x_1$, and $y$
            demonstrating the degeneracy of $V \wedge V \to \lie{n}^2$.
            This basis demonstrates $\lie{n} \cong \R \oplus \lie{n}_3$.
        \item If $\dim \lie{n}^2 = 0$, then $\lie{n} \cong \R^4$.
    \end{itemize}
\end{proof}

\subsubsection{\texorpdfstring{$\R^4$}{R4} semidirect sums}
\label{semiwithabelian}

Suppose $G$ is a model geometry and $\lie{g} = \lie{n} \semisum_\phi \R$.
If $\lie{n} = \R^4$, a linear change of coordinates puts
the image of $1$ under $\phi: \R \to \lie{sl}_4 \R$ in Jordan form.
Listing Jordan forms, grouped by number of blocks, yields the following.
(Omitted entries are zero,
and ``$*$'' entries are subject only to the restriction that the
whole matrix is traceless.)
\begin{align*}
    \phi_{1}(1) &= \begin{pmatrix}
        0 & 1 & & \\ & 0 & 1 & \\ & & 0 & 1 \\ & & & 0
    \end{pmatrix} &
    \phi_{1'}(1) &= \begin{pmatrix}
        & \lambda & 1 & \\ -\lambda & & & 1 \\
        & & & \lambda \\ & & -\lambda &
    \end{pmatrix} \\
    \phi_{2}(1) &= \begin{pmatrix}
        \lambda & 1 & & \\ & \lambda & & \\
        & & -\lambda & 1 \\ & & & -\lambda
    \end{pmatrix} &
    \phi_{2'}(1) &= \begin{pmatrix}
        \lambda & 1 & & \\ & \lambda & 1 & \\
        & & \lambda & \\ & & & -3\lambda
    \end{pmatrix} \\
    \phi_{2''}(1) &= \begin{pmatrix}
        \lambda & \mu & & \\ -\mu & \lambda & & \\
        & & -\lambda & 1 \\ & & & -\lambda
    \end{pmatrix} &
    \phi_{2'''}(1) &= \begin{pmatrix}
        \lambda & \mu_1 & & \\ -\mu_1 & \lambda & & \\
        & & -\lambda & \mu_2 \\ & & -\mu_2 & -\lambda
    \end{pmatrix} \\
    \phi_{3}(1) &= \begin{pmatrix}
        \lambda & 1 & & \\ & \lambda & & \\ & & * & \\ & & & *
    \end{pmatrix} &
    \phi_{3'}(1) &= \begin{pmatrix}
        \lambda & \mu & & \\ -\mu & \lambda & & \\ & & * & \\ & & & *
    \end{pmatrix} \\
    \phi_4(1) &= \begin{pmatrix}
        * & & & \\ & * & & \\ & & * & \\ & & & *
    \end{pmatrix} \\
\end{align*}
With this, it becomes possible to classify maximal model geometries
of the form $\R^4 \rtimes \R$;
computing characteristic polynomials
recovers the list in Thm.~\ref{thm:main}(ii)(b)
from the following statement.
\begin{prop}[\textbf{Classification of $\R^4 \rtimes \R$ geometries}]
    \label{prop:semiwithabelian}
    ~
    \begin{enumerate}[(i)]
        \item If $G = G/\{1\}$ is a maximal model geometry with Lie algebra
            $\lie{g} = \R^4 \semisum_\phi \R$, then $\phi$ can be taken
            to be one of
            $\phi_1$,
            $\phi_2$ ($\lambda = 1$),
            $\phi_{2'}$ ($\lambda = 0$),
            $\phi_3$ ($\lambda = 0$), or
            $\phi_4$.
        \item All cases listed in (i) are model geometries,
            except that the group with Lie algebra $\R^4 \semisum_{\phi_4} \R$
            is a model geometry if $\op{exp} \phi_4(t)$
            has integer characteristic polynomial for some $t \neq 0$.
    \end{enumerate}
\end{prop}

Note that some of the actions $\phi_i$
depend on parameters, and not all of the parameter values produce model geometries.
Fortunately, there is an easy-to-state necessary and sufficient
condition for model geometries arising in this case:
recall that $G/G_p$ is a model geometry if and only if some lattice
$\Gamma \subset G$ intersects no conjugate of $G_p$ nontrivially
(Prop.~\ref{prop:geometries3}(ii));
so in the case of $G_p = \{1\}$,
this reduces to the question of whether a lattice exists,
which is determined by the following condition.
\begin{lemma}[\hspace{1sp}{\cite[Cor.~6.4.3]{filipk}}] \label{mostowcor}
    A unimodular, non-nilpotent $\R^n \rtimes_{\op{exp} \phi} \R$
    admits a lattice if and only if
    there is $0 \neq t \in \R$ such that
    the characteristic polynomial of $\op{exp} \phi(t)$
    has coefficients in $\Z$.
\end{lemma}

\begin{proof}[Proof of Prop.~\ref{prop:semiwithabelian}]
    The complete list of actions of $\R$ on $\R^4$ would produce
    a long list of cases, so a first step will be to eliminate
    actions producing non-maximal geometries.
    Each remaining group is then examined to determine whether
    it admits a lattice. The cases are grouped by the number of Jordan blocks.

    \paragraph{Preparatory step: ignore non-maximal geometries.}
    If $\phi(1)$ has non-real eigenvalues,
    then it (and the $1$-parameter subgroup of $\op{SL}(4,\R)$ it generates)
    commutes with rotations on some $2$-dimensional eigenspace.
    These rotations form an $S^1 \subseteq \op{Aut}(G)$;
    so $G/\{1\}$ is not maximal, since it is
    subsumed by $G \rtimes S^1 / S^1$.
    This eliminates $\phi_{1'}$, $\phi_{2''}$, $\phi_{2'''}$, and $\phi_{3'}$.

    \paragraph{Case 1: $\phi(1)$ has 1 Jordan block.}
    Then $\phi = \phi_1$. Since
        \[ \op{exp} \phi_1(6) = \begin{pmatrix}
            1 & 6 & 18 & 36 \\ & 1 & 6 & 18 \\ & & 1 & 6 \\ & & & 1
        \end{pmatrix}\]
    has integer entries,
    its characteristic polynomial has integer coefficients;
    so the resulting $\R^4 \rtimes \R$ admits a lattice
    by Lemma \ref{mostowcor} above.
    In fact the lattice can be taken to be a subgroup isomorphic to
    $\Z^4 \rtimes_{\op{exp} \phi_1(6)} \Z$,
    with fundamental domain
    $\left\{\left((\op{exp} \phi_1(t))x, t\right)
    \mid t \in (0,6), x \in [0,1]^4
    \right\}$.

    \paragraph{Case 2a: $\phi = \phi_2$ (2 Jordan blocks).}
    If $\lambda = 0$ then $\phi_2(t)$ commutes
    with an $S^1$ of rotations
    (it coincides with the $\lambda = 0$ case for $\phi_{1'}$)
    and will not produce a maximal geometry.

    If instead $\lambda \neq 0$,
    then by rescaling $\phi_2$ and changing basis to put the new $\phi_2(1)$
    in Jordan form, we can assume $\lambda = 1$, so this produces at most
    one new geometry. Exponentiating yields the $1$-parameter subgroup
        \[ \left\{ \begin{pmatrix}
            e^t & te^t & & \\
            & e^t & & \\
            & & e^{-t} & te^{-t} \\
            & & & e^{-t}
        \end{pmatrix} : t \in \R \right\} \subset \op{SL}(\R^4) . \]
    Reordering the basis elements turns this into the block matrices
        \[
            \left\{ \begin{pmatrix}
                A(t) & tA(t) \\ & A(t)
            \end{pmatrix} : t \in \R,
            A(t) = \begin{pmatrix} e^t & \\ & e^{-t} \end{pmatrix} \right\}
            \subset \op{SL}(\R^4) .
        \]
    If $B$ diagonalizes $C = \begin{pmatrix} 2 & 1 \\ 1 & 1 \end{pmatrix}$, then
    after conjugating by the block diagonal matrix with diagonal blocks $B^{-1}$,
    this $1$-parameter subgroup contains an element of the form
    $\begin{pmatrix} C & sC \\ & C \end{pmatrix}$, with $s \in \R$.
    After rescaling the first two basis elements,
    we conclude as in Case 1 above that
    the Lie group $\R^4 \rtimes_{\op{exp} \phi_2} \R$
    admits a lattice isomorphic to $\Z^4 \rtimes_{A'} \Z$, where
        \[ A' = \begin{pmatrix} 2 & 1 & 2 & 1 \\ 1 & 1 & 1 & 1 \\
            & & 2 & 1 \\ & & 1 & 1 \end{pmatrix} .\]

    \paragraph{Case 2b: $\phi = \phi_{2'}$ (2 Jordan blocks).}
    We will prove in this case that the resulting $G = \R^4 \rtimes \R$
    is a model geometry if and only if $\lambda = 0$.

    By Lemma \ref{mostowcor} above, if $G$ is a model geometry
    then $\op{exp} \phi_{2'}(t)$ has characteristic
    polynomial $p(x) \in \Z[x]$ for some $t \neq 0$.
    Then $e^{t\lambda}$ is a triple root of $p$, so its minimal polynomial
    over $\Q$ divides $p$ at least $3$ times.
    Since $\deg p = 4$, the minimal polynomial of
    $e^{t\lambda}$ is linear; so $e^{t\lambda} \in \Q$.
    Since $e^{t\lambda}$ is a rational root of a polynomial whose first
    and last coefficients are $1$, the rational root theorem
    implies $e^{t\lambda} = \pm 1$. Then since $t \neq 0$
    and $\lambda$ is real, $\lambda = 0$.

    Conversely, $\lambda = 0$ then $\op{exp} \phi_{2'}(2)$ has integer
    entries; so $G$ admits the lattice $\Z^4 \rtimes_{\op{exp} \phi_{2'}(2)} \Z$.
    In the notation of \cite[\S 6.4.6]{filipk}, this gives the geometry $G_3 \times \R$.

    \paragraph{Case 3: $\phi = \phi_3$ (3 Jordan blocks).}
    Again we prove that the resulting $\R^4 \rtimes \R$
    is a model geometry if and only if $\lambda = 0$;
    moreover, this produces only one geometry.

    If a model geometry results, then $\op{exp} \phi_3(t)$
    has characteristic polynomial $p(x) \in \Z[x]$ for some $t \neq 0$.
    Then $e^{t\lambda}$ is a double root of $p(x)$,
    so its minimal polynomial over $\Q$ divides $p(x)$ at least twice.
    If $e^{t\lambda} \notin \Q$, then the two eigenvalues of $\phi_3(1)$
    that aren't $\lambda$ are identical; so the last two coordinates
    admit an $S^1$ of rotations, making $G$ non-maximal.
    Then $e^{t\lambda} \in \Q$;
    so as in Case 2b, $\lambda = 0$ by the rational root theorem.

    If $\lambda = 0$, then $\phi_3$
    and the first basis vector of $\R^4$ can be rescaled so that
        \[ \phi_3(1) = \begin{pmatrix}
            0 & 1 & & \\ & 0 & & \\ & & \alpha & \\ & & & -\alpha
        \end{pmatrix} \]
    where $\alpha = \ln \frac{3 + \sqrt{5}}{2}$.
    (If the diagonal entries were all zero,
    then the last two coordinates would admit an $S^1$ of rotations.)
    Exponentiating this yields a $1$-parameter subgroup of $\op{SL}(\R^4)$
    containing a matrix $A$ that in some basis becomes
        \[ \begin{pmatrix}
            1 & 1 & & \\ & 1 & & \\ & & 2 & 1 \\ & & 1 & 1
        \end{pmatrix} , \]
    so $G$ admits a lattice isomorphic to $\Z^4 \rtimes_A \Z$.

    \paragraph{Case 4: $\phi = \phi_4$ (4 Jordan blocks).}
    The criterion claimed in (ii) for producing a model geometry
    is merely Lemma \ref{mostowcor} above.
    This completes the proof of Prop.~\ref{prop:semiwithabelian}.
\end{proof}

\begin{rmk}[\textbf{Alternative parametrization of the $\phi = \phi_4$ case}]
    The $3$-dimensional family of maps $\phi_4$
    produces a $2$-dimensional family of groups $\R^4 \rtimes_{\op{exp} \phi_4} \R$
    (due to the ability to rescale the $\R$ factor),
    but not all are model geometries---by Lemma \ref{mostowcor} above,
    $\op{exp} \phi_4(t)$ has to have an integer characteristic polynomial
    $p(x) = x^4 + ax^3 + bx^2 + cx + 1 \in \Z[x]$
    for some nonzero $t$.
    Allowing for rescaling the $\R$ factor,
    one could instead parametrize this family of geometries
    by the coefficients $(a,b,c) \in \Z^3$; but the correspondence
    is not a bijection due to the following.
    \begin{enumerate}[(i)]
        \item Since $\phi_4(t)$ has $4$ real eigenvalues,
            $p(x)$ must have $4$ nonnegative real roots.
        \item Since duplicate eigenvalues of $\phi_4(t)$ would
            make $\op{exp} \phi_4(t)$ commute with an $S^1$ of rotations,
            the roots of $p(x)$ must be distinct
            if it comes from a maximal geometry.
        \item If $\op{exp} \phi_4(t)$ has integer characteristic polynomial
            (i.e.\ is an integer matrix in some basis),
            then so does $\op{exp} \phi_4(nt)$ for any $n \in \Z$;
            so each geometry corresponds to a $\Z$-family of polynomials.
    \end{enumerate}
    In \cite[\S{6.4.8}]{filipk},
    Filipkiewicz gave a detailed description of this kind
    of parametrization for the analogous family $\Sol^4_{m,n}$
    in dimension $4$.

\end{rmk}

\subsubsection{\texorpdfstring{$\lie{n}_4$}{n4} semidirect sums}

The classification of $\lie{n}_4 \semisum \R$ geometries
has a slightly different flavor from the $\R^4 \semisum \R$ case---there
are fewer derivations of $\lie{n}_4$ than of $\R^4$;
but the tradeoff is that they require a bit more work to find.
The result is the following.

\begin{prop}[\textbf{Classification of $\lie{n}_4 \semisum \R$ geometries}] \label{seminilfour}
    Suppose $G = G/\{1\}$ is a model geometry whose Lie algebra $\lie{g}$ is of the form
    $\lie{n}_4 \semisum \R$ and not expressible as $\R^4 \semisum \R$. Then
    $\lie{g}$ has basis $\{x_1,x_2,x_3,x_4,x_5\}$ with
    \begin{align*}
        [x_4, x_2] &= x_1 &
        [x_4, x_3] &= x_2 &
        [x_5, x_3] &= x_1 &
        [x_5, x_4] &= 0 \text{ or } x_3 ,
    \end{align*}
    and all other brackets not determined by skew-symmetry are zero.
    Both Lie algebras thus described are the Lie algebras of model geometries.
\end{prop}

Most of the proof lies in describing a derivation $D$ that makes
$\lie{n}_4 \semisum_D \R$ the Lie algebra of a model geometry.
Predictably, this case begins with the computation of
the traceless outer derivation algebra $\op{sout}(\lie{n}_4)$
(Lemma \ref{lemma:n4_derivations}). This will be followed by an extra condition
on $D$ for model geometries (Lemma \ref{lemma:semisum_center_traceless})
and  the proof of Prop.~\ref{seminilfour}.
\begin{lemma} \label{lemma:n4_derivations}
    Every element of $\op{sout}(\lie{n}_4)$
    is represented by a matrix of the form
    \[
        D_{4,a,b,c} = \begin{pmatrix}
            2a & 0 & b & 0 \\
            0 & -a & 0 & 0 \\
            0 & 0 & -4a & c \\
            0 & 0 & 0 & 3a
        \end{pmatrix} ,
    \]
    with respect to the basis in Definition \ref{defn:nildefault}.
\end{lemma}
\begin{proof} Suppose $D \in \op{der} \lie{n}_4$ is traceless.
    $\op{ad} \lie{n}_4$ consists of the matrices
        \[ \begin{pmatrix}
            0 & c & 0 & a \\
            0 & 0 & c & b \\
            0 & 0 & 0 & 0 \\
            0 & 0 & 0 & 0
        \end{pmatrix} . \]
    Any derivation $D$ is upper triangular since it preserves the following
    filtration by characteristic ideals.
        \begin{align*}
            \vecspan{x_1} &= Z(\lie{n}) \\
            \vecspan{x_1, x_2} &= \lie{n}_4^2 \\
            \vecspan{x_1, x_2, x_3} &= \{x \in \lie{n}_4 \mid \dim [x,\lie{n}_4] < 2 \}
        \end{align*}
    Hence, up to an inner derivation, we may write $D$ as
        \begin{align*}
            Dx_1 &= a_1 x_1 \\
            Dx_2 &= b_2 x_2 \\
            Dx_3 &= c_1 x_1 + c_2 x_2 + c_3 x_3 \\
            Dx_4 &= d_3 x_3 + d_4 x_4 .
        \end{align*}
    Using the Leibniz rule,
        \begin{align*}
            a_1 x_1 &= Dx_1 = [Dx_4, x_2] + [x_4, Dx_2] = (d_4 + b_2) x_1 \\
            b_2 x_2 &= Dx_2 = [Dx_4, x_3] + [x_4, Dx_3] = (d_4 + c_3) x_2 + c_2 x_1 ,
        \end{align*}
        so $c_2 = 0$ and $c_3 = b_2 - d_4 = a_1 - 2d_4$.
    Finally, if $D$ is traceless,
        \[ 0 = a_1 + b_2 + c_3 + d_4 = 3c_3 + 4d_4 . \qedhere \]
\end{proof}

\begin{lemma} \label{lemma:semisum_center_traceless}
    Suppose $G = G/\{1\}$ is a model geometry whose Lie algebra
    $\lie{g}$ is of the form $\lie{n} \semisum \R$ where $\lie{n}$ is nilpotent.
    Then $\R$ acts tracelessly on $Z(\lie{n})$.
\end{lemma}
\begin{proof}
    Since $Z(\lie{n})$ is a characteristic ideal, it is stable
    under the action of $\R$.
    Suppose $\R$ acts with nonzero trace on $Z(\lie{n})$.

    Then $\R$ acts non-nilpotently,
    so $\lie{n} = \nilrad \lie{g}$.
    Write $G = N \rtimes \R$ where $N$ is the simply-connected Lie group
    with Lie algebra $\lie{n}$.

    Since $G$ is a model geometry, it admits a lattice $\Gamma$.
    The nilradical of a solvable Lie group inherits lattices
    \cite[Lemma 3.9]{mostow}---that is,
    $N \cap \Gamma$ is a lattice in $N$,
    and $\Gamma/(N \cap \Gamma)$ is a lattice in $\R$.
    The latter implies that some element of $\Gamma$
    acts by conjugation with determinant $> 1$ on $Z(N)$.

    Since $Z(N)$ is a term in the upper central series of $N$,
    the intersection $Z(N) \cap \Gamma$ is a lattice in $Z(N)$
    \cite[Prop.~2.17]{raghunathan}.
    This is impossible since
    no lattice of $\R^k$ is stable under a linear map
    with determinant $> 1$.
\end{proof}

Given the above, the classification of model geometries with tangent algebra
of the form $\lie{n}_4 \semisum \R$ is within reach and will recover the list
in Thm.~\ref{thm:main}(ii)(c).
\begin{proof}[Proof of Prop.~\ref{seminilfour}]
    The action of $\R$ on $\lie{n}_4$ is represented by a derivation
    of the form described above in Lemma \ref{lemma:n4_derivations}.
    The following steps recover the Lie algebra defined
    in the statement of Prop.~\ref{seminilfour}.
    \begin{enumerate}
        \item Since $G$ needs to be a model geometry,
            $a = 0$ by the tracelessness condition
            from Lemma \ref{lemma:semisum_center_traceless}.
        \item If $b = 0$, then $x_1$, $x_2$, $x_3$, and $x_5$ span an abelian
            ideal, so $G$ is also expressible as $\R^4 \rtimes \R$. Hence
            $b \neq 0$. Rescale $x_5$ by a factor of $b^{-1}$---i.e.\ replace $x_5$
            with $b^{-1} x_5$---to make $b = 1$.
        \item If $c \neq 0$, then similarly rescale
            $\vecspan{x_1, x_2, x_3}$ by a factor of $c$ to make $c = 1$.
    \end{enumerate}
    In the end only one parameter can vary---$c$ may be $0$ or $1$.
    Either way, $\lie{g}$ is nilpotent and expressed in a basis
    with integral structure constants; so the corresponding simply-connected
    group $G$ admits a lattice \cite[Thm.~2.12]{raghunathan}.
    Therefore $G$ is a model geometry.
\end{proof}

\subsubsection{\texorpdfstring{$\R \oplus \lie{n}_3$}{R + n3} semidirect sums}
\label{heisenplusline}

Despite only producing two geometries,
this case is the messiest of the three---$\R \oplus \lie{n}_3$
has enough outer derivations to make this a long problem,
yet not so many that a systematic approach is as easy as
listing Jordan blocks in $\lie{sl}_4 \R$.
The main result is the following.

\begin{prop}[\textbf{Classification of $(\R \times \Heis_3) \semisum \R$ geometries}]
    \label{prop:heisenpluslorentz}
    Suppose $G = G/\{1\}$ is a maximal model geometry with Lie algebra $\lie{g}$
    of the form $(\R \oplus \lie{n}_3) \semisum_D \R$.
    If $\lie{g}$ contains no ideal isomorphic to $\R^4$ or $\lie{n}_4$, then
    $\lie{g}$ is isomorphic to a Lie algebra constructed by letting
    $1 \in \R$ act on $\R \oplus \lie{n}_3 = \vecspan{y,x_1,x_2,x_3}$ by
    \begin{align*}
        x_2 &\mapsto x_2 &
        x_3 &\mapsto -x_3 &
        y &\mapsto 0 \text{ or } x_1 .
    \end{align*}
    For both Lie algebras thus named,
    the corresponding simply-connected group $G$ is a model geometry.
\end{prop}

The strategy will broadly be the same as before: list
elements of $\op{sout} (\R \oplus \lie{n}_3)$ by which $1 \in \R$
can act; use coordinate changes to cluster the resulting Lie algebras
by isomorphism type; and show that each admits a lattice.
To save some work, we begin with a lemma that eliminates
some of $\op{sout} (\R \oplus \lie{n}_3)$ from consideration.

\begin{lemma}
    Let $\lie{n} = \R \oplus \lie{n}_3$.
    If $G = G/\{1\}$ is a model geometry whose Lie algebra $\lie{g}$
    is of the form $\lie{n} \semisum \R$,
    then $\R$ acts tracelessly on $\lie{n}^2 = Z(\lie{n}_3)$.
\end{lemma}
\begin{proof}
    Suppose $\R$ acts with nonzero trace on the characteristic ideal $Z(\lie{n}_3)$.
    Let $\Gamma$ be a lattice in $G = (\R \times \Heis_3) \rtimes \R$.
    Following the argument in Lemma \ref{lemma:semisum_center_traceless},
    some element of $\Gamma$ acts with determinant $< 1$ on $Z(\Heis_3)$, and
    $\Gamma \cap (\R \times \Heis_3)$ maps to a lattice in
    $(\R \times \Heis_3)/Z(\R \times \Heis_3) \cong \R^2$.

    Then over some two linearly independent elements of $\R^2$
    lie two elements of $\Gamma$; their commutator is a nontrivial
    element of $Z(\Heis_3) \cong \R$.
    This is impossible since no discrete subgroup of $\R$
    is stable under an automorphism with determinant $< 1$.
\end{proof}

In combination with the requirement that $\R$ acts
tracelessly on $Z(\lie{n})$ (Lemma \ref{lemma:semisum_center_traceless}),
this brings the dimension of the relevant subspace of $\op{sout} (\R \oplus \lie{n}_3)$
down to $6$.
\begin{lemma} \label{lemma:heisenplus_sout}
    If $\lie{n} = \R \oplus \lie{n}_3$ has ordered basis $(y, x_1, x_2, x_3)$
    (with $[x_3, x_2] = x_1$ as in Definition \ref{defn:nildefault}),
    then every element of $\op{sout}(\lie{n})$
    which also acts tracelessly on $\lie{n}^2 = \vecspan{x_1}$ and
    $Z(\lie{n}) = \vecspan{x_1, y}$ is represented
    by a matrix of the form
    \[
        \begin{pmatrix}
            0 & 0 & e & f \\
            d & 0 & 0 & 0 \\
            0 & 0 & a & b \\
            0 & 0 & c & -a
        \end{pmatrix} .
    \]
\end{lemma}
\begin{proof}
    The tracelessness assumptions, along with the fact that
    derivations preserve characteristic ideals,
    imply that $D x_1 = 0$, that $D y \in \vecspan{x_1}$,
    and that the last two diagonal entries add to $0$.

    By subtracting multiples of the inner derivations
    $\op{ad} x_2$ and $\op{ad} x_3$, we can assume
    that $D x_2$ and $D x_3$ have zero $x_1$ coordinate.
\end{proof}

The proof of Prop.~\ref{prop:heisenpluslorentz} which now follows
is similar in spirit to that for $\lie{n}_4$ (Prop.~\ref{seminilfour})---it
shows that most of this space of derivations either produces no
maximal model geometries or produces geometries accounted for by previous cases.
\begin{proof}[Proof of Prop.~\ref{prop:heisenpluslorentz}]
    Give $\lie{g}$ the ordered basis $(y,x_1,x_2,x_3,z)$, where the
    first four elements are as above in Lemma \ref{lemma:heisenplus_sout}
    and $z$ spans the last $\R$ factor.
    Let $D \in \op{der} (\R \oplus \lie{n}_3)$ be the derivation
    by which $z$ acts;
    and let $D'$ be the lower right $2 \times 2$ block of $D$.

    \paragraph{Preparatory step: Simplify $D$ with coordinate changes.}
    By changing basis in $\vecspan{x_2, x_3}$, we can assume
    $D'$ is in Jordan form. By rescaling
    $z$, we may assume this block contains only $1$, $0$, and $-1$
    as entries.

    For $r \neq 0$, let $\mu_r: \R \oplus \lie{n}_3 \to \R \oplus \lie{n}_3$
    be the automorphism given by
    \begin{align*}
        x_1 &\mapsto r^2 x_1 &
        x_2 &\mapsto r x_2 &
        x_3 &\mapsto r x_3 &
        y &\mapsto r y .
    \end{align*}
    In the notation of Lemma \ref{lemma:heisenplus_sout},
    conjugating $D$ by $\mu_r$ replaces $d$ by $rd$
    while leaving the other entries unchanged;
    so we may also assume that $d$ will always be $0$ or $1$.

    \paragraph{Case 1: If $D' = 0$ then no new maximal geometries arise.}
    In this case, $D'$ commutes with the automorphisms of $\R \oplus \lie{n}_3$
    acting as rotations on $\vecspan{x_2,x_3}$.
    In the notation of Lemma \ref{lemma:heisenplus_sout},
    if $e = f = 0$, then these automorphisms also commute with $D$,
    so they extend to automorphisms of $\lie{g}$.
    This would make $G$ nonmaximal, since $G \cong G \rtimes S^1/S^1$.

    Thus $e$ and $f$ are not both zero;
    so we can conjugate by a rotation to make $f = 0$
    and a rescaling of $y$ to make $e = 1$.
    Under this modified basis, $D$ has the matrix
    \begin{align*}
        \begin{pmatrix}
            0 & & 1 & \\
            d & 0 & & \\
            & & 0 & \\
            & & & 0
        \end{pmatrix},
    \end{align*}
    where $d$ is $0$ or $1$. In either case, $\lie{g}$
    has an ideal isomorphic to $\R^4$ or $\lie{n}_4$:
    \begin{itemize}
        \item If $d = 0$, then $\vecspan{y, x_1, x_3, z} \subset \lie{g}$
            is a $4$-dimensional abelian ideal.
        \item If $d = 1$, then $\vecspan{x_1, y, x_2, z} \subset \lie{g}$
            is an ideal isomorphic to $\lie{n}_4$, with the basis ordered
            as in Definition \ref{defn:nildefault}---that is, the nonzero
            brackets are
            \begin{align*}
                [z,y] &= x_1  &
                [z,x_2] &= y .
            \end{align*}
    \end{itemize}

    \paragraph{Case 2: If $D' \neq 0$ is skew-symmetric then $G$ is non-maximal.}
    As in Case 1, $D'$ commuting with a compact group of automorphisms
    leads to the conclusion that $D$ has matrix
    \begin{align*}
        \begin{pmatrix}
            0 & & 1 & \\
            d & 0 & & \\
            & & & 1 \\
            & & -1 &
        \end{pmatrix}.
    \end{align*}
    Define
    \begin{align*}
        u_1 &= x_1 &
        u_2 &= x_2 &
        u_3 &= x_3 - y &
        v &= y &
        w &= z - dx_2 .
    \end{align*}
    Then $\vecspan{v, u_1, u_2, u_3} \cong \R \oplus \lie{n}_3$, and
    $w$ acts in this ordered basis by the matrix
    \begin{align*}
        \begin{pmatrix}
            0 & & & \\
            d & 0 & & \\
            & & & 1 \\
            & & -1 &
        \end{pmatrix},
    \end{align*}
    which commutes with a compact group of automorphisms,
    implying as in Case 1 that $G$ is non-maximal.

    \paragraph{Case 3: If $D'$ is a single Jordan block
        then no new maximal geometries arise.}
    In the notation of Lemma \ref{lemma:heisenplus_sout},
    either $e = 0$, or we can replace $x_3$
    by $x_3 - \frac{f}{e} x_2$ to make $f = 0$.
    If one of $e$ and $f$ remains nonzero, then $y$ can be rescaled
    to make it equal $1$. This produces three matrices for $D$:
    \begin{align*}
        D_{1,d,0} &= \begin{pmatrix}
            0 &   &   & \\
            d & 0 &   & \\
              &   & 0 & 1 \\
              &   &   & 0
        \end{pmatrix} &
        D_{1,d,1} &= \begin{pmatrix}
            0 &   & 1 & \\
            d & 0 &   & \\
              &   & 0 & 1 \\
              &   &   & 0
        \end{pmatrix} &
        D_{1,d,2} &= \begin{pmatrix}
            0 &   &   & 1 \\
            d & 0 &   & \\
              &   & 0 & 1 \\
              &   &   & 0
        \end{pmatrix}
    \end{align*}
    In all but one case, an ideal isomorphic to $\lie{n}_4$ can be named
    by an ordered basis $(v_0, v_1, v_2, v_3)$ where the nonzero brackets are
    $[v_3, v_2] = v_1$ and $[v_3, v_1] = v_0$.
    \begin{align*}
        D_{1,d,0} &: x_1, x_2, -z, x_3 \\
        D_{1,1,1} &: x_1, y, x_2, z \\
        D_{1,d,2} &: x_1, x_2 + y, dx_3 - z, x_3
    \end{align*}

    If $D = D_{1,0,1}$, then $\lie{g}$ has basis $\{y,x_1,x_2,x_3,z\}$
    with brackets
    \begin{align*}
        [z,x_2] &= y  &
        [x_3,x_2] &= x_1  &
        [z,x_3] &= x_2  .
    \end{align*}
    In this case $\lie{g}$ admits a compact group of automorphisms $\phi_\theta$ defined by
    \begin{align*}
        z &\mapsto z \cos \theta + x_3 \sin \theta  &
        y &\mapsto y \cos \theta + x_1 \sin \theta  &
        x_2 &\mapsto x_2  \\
        x_3 &\mapsto -z \sin \theta + x_3 \cos \theta  &
        x_1 &\mapsto -y \sin \theta + x_1 \cos \theta ,
    \end{align*}
    so $G/\{1\}$ is non-maximal when $D = D_{1,0,1}$.

    \paragraph{Case 4: If $D' \neq 0$ is diagonal
        then $\lie{g}$ is determined up to isomorphism.}
    In this case, $D$ is the matrix
    \begin{align*}
        \begin{pmatrix}
            0 & & e & f \\
            d & 0 & & \\
            & & 1 & \\
            & & & -1
        \end{pmatrix}.
    \end{align*}
    Let
    \begin{align*}
        u_1 &= x_1 &
        u_2 &= x_2 + ey &
        u_3 &= x_3 - fy &
        v &= y &
        w &= z - dex_3 - dfx_2 .
    \end{align*}
    Then $u_1$ is central, $v$ commutes with all $u_i$, and
    \begin{align*}
        [u_3, u_2] &= u_1 \\
        [w, u_2] &= u_2 \\
        [w, u_3] &= -u_3 \\
        [w, v] &= du_1 .
    \end{align*}
    This shows that $e$ and $f$ do not affect the isomorphism type of $\lie{g}$;
    taking $e = f = 0$ recovers the
    the definition in the statement of Prop.~\ref{prop:heisenpluslorentz}.

    \paragraph{Final step: Verify these are model geometries}
    A finite-volume quotient of $G$ by a subgroup of its isometry group
    is a quotient by a lattice (Prop.~\ref{prop:geometries3}(ii));
    so it suffices to find a lattice in $G$.

    Suppose $s \in \R$ is nonzero.
    Putting coordinates $(y, x_1, x_2, x_3)$ on $\R \times \Heis_3$,
    the semidirect product $(\R \times \Heis_3) \rtimes_{e^{tA}} \R$ where
    \begin{align*}
        A &= \begin{pmatrix}
            0 & & & \\
            s^{-1} d & 0 & & \\
            & & -1 & 2 \\
            & & 2 & 1
        \end{pmatrix}
    \end{align*}
    is, by the preparatory step, isomorphic to $G$.
    In particular, if $s = \dfrac{\ln \frac{3 + \sqrt{5}}{2}}{\ln \sqrt{5}}$, then
    \begin{align*}
        e^{sA} &= \begin{pmatrix}
            1 & & & \\
            d & 1 & & \\
            & & 1 & 1 \\
            & & 1 & 2
        \end{pmatrix}.
    \end{align*}
    This is an integer matrix, so it preserves the lattice $\Gamma$
    in $\R \times \Heis_3$ consisting of the integer points.
    Then $\Gamma \rtimes_{e^{sA}} \Z$ is a lattice in $G$,
    with closed fundamental domain the standard unit cube of coordinates
    $[0,1]^5$.
\end{proof}

\subsection{Maximality and distinctness}
\label{sec:solvable_maximality}

This section proves that the geometries $G = G/\{1\}$ obtained above
are both maximal and distinct. Distinctness (Prop.~\ref{prop:solvable_distinctness})
mostly uses dimensions of characteristic subalgebras to distinguish
the corresponding Lie algebras from each other,
and is summarized in Figure \ref{fig:solvable_flowchart}.
Maximality (Prop.~\ref{prop:solvable_maximality})
uses the following theorem by Gordon and Wilson.\footnote{
    This generalizes \cite[Thm.~2(3)]{wilson1982isometry},
    which is the same result for nilpotent Lie groups.
}

\begin{thm}[part of {\cite[Thm.~4.3]{gordonwilsonisometry}}]
    \label{thm:gordonwilsonisometry}
    Suppose $G$ is a connected unimodular solvable Lie group with Lie algebra $\lie{g}$.
    If the elements of $\op{ad} \lie{g}$ have only real eigenvalues
    (i.e.\ $\lie{g}$ has all real roots), then in any invariant metric,
    there is some $C \subseteq \op{Aut} G$ such that $(\op{Isom} G)^0 = G \rtimes C$.
\end{thm}

Then $G \cong G \rtimes C / C$, and $C$ becomes the stabilizer of the identity in $G$;
so for a maximal geometry, $C$ is a maximal compact subgroup of $\op{Aut} G$.
The following computation helps to find the maximal compact subgroups.
\begin{lemma} \label{lemma:compacts_in_aut}
    Suppose $H \subseteq \op{GL}(n,\R)$ is a Lie group of consisting of
    block upper-triangular matrices, where $k_1,\ldots,k_m$ are the block sizes.
    Then its maximal compact subgroup is conjugate in $\op{GL}(n,\R)$
    to a subgroup of $\op{SO}(k_1) \times \cdots \times \op{SO}(k_m)$.
\end{lemma}
\begin{proof}
    Let $G(k_1,\ldots,k_m)$ be the group of block upper-triangular matrices
    where $k_1,\ldots,k_m$ are the block sizes. Since the maximal compact subgroup of $H$
    is a compact subgroup of $G(k_1,\ldots,k_m)$, it suffices to compute
    the maximal compact subgroup of $G(k_1,\ldots,k_m)$.

    The determinant of a block upper-triangular matrix is the product of the
    determinants of the blocks. The straight-line homotopy that sends
    matrix entries outside the blocks to zero also preserves the determinant,
    thereby specifying a deformation retract of $G(k_1,\ldots,k_m)$
    onto the corresponding group of block diagonal matrices,
    which deformation retracts onto $\op{SO}(k_1) \times \cdots \times \op{SO}(k_m)$.
\end{proof}

\begin{prop} \label{prop:solvable_maximality}
    Let $G$ be the simply-connected Lie group with one of the
    following solvable Lie algebras $\lie{g}$.
    Then $G/\{1\}$ is a maximal geometry.
    \begin{itemize}
        \item (nilpotent) $\R^4 \semisum \R$ where $1 \in \R$ acts with Jordan blocks
            with characteristic polynomials $(x^4)$ or $(x^3, x)$.
        \item (nilpotent) $\lie{n}_4 \semisum \R$ where $1 \in \R$ acts by
            $x_3 \mapsto x_1$ and $x_4 \mapsto 0$ or $x_3$ (as named in Prop.\ \ref{seminilfour}).
        \item $\R^3 \semisum \R^2$ where $\R^2$ acts by traceless diagonal matrices.
        \item $\R^4 \semisum \R$ where $1 \in \R$ acts diagonalizably
            with distinct eigenvalues or with Jordan blocks
            with characteristic polynomials
            $(x^2, x-1, x+1)$ or $((x-1)^2, (x+1)^2)$.
        \item $(\R \oplus \lie{n}_3) \semisum \R$ (as named in Prop.\ \ref{prop:heisenpluslorentz}) where $1 \in \R$ acts by
            \begin{align*}
                x_2 &\mapsto x_2  &  x_3 &\mapsto -x_3  &  y \mapsto 0 \text{ or } x_1 .
            \end{align*}
    \end{itemize}
\end{prop}
\begin{proof}
    Using Thm.~\ref{thm:gordonwilsonisometry} (the Gordon-Wilson result)
    requires showing that
    $\lie{g}$ is unimodular---which is already known since
    $G$ must admit a lattice to be a geometry---and
    that the eigenvalues of its adjoint representation are all real.

    \paragraph{Case 1: $G$ is nilpotent.}
    For a nilpotent $\lie{g}$, the adjoint representation's eigenvalues are all zero;
    so showing that the maximal compact connected subgroup of $\op{Aut} G$
    is trivial will show that $G$ is maximal.
    By Lemma \ref{lemma:compacts_in_aut} above,
    it suffices to show that each Lie algebra
    has a complete flag of characteristic ideals. They are as follows.
    \begin{itemize}
        \item If $\lie{g} = \R^4 \semisum \R [x^4]$:
            \begin{align*}
                \vecspan{x_1} &= \lie{g}^4  \\
                \vecspan{x_2, x_1} &= \lie{g}^3  \\
                \vecspan{x_3, x_2, x_1} &= \lie{g}^2  \\
                \vecspan{x_4, x_3, x_2, x_1}
                    &= \{ x \in \lie{g} \mid \dim [x, \lie{g}] \leq 1 \}
            \end{align*}
        \item If $\lie{g} = \R^4 \semisum \R [x^3, x]$:
            \begin{align*}
                \vecspan{x_1} &= \lie{g}^3  \\
                \vecspan{x_2, x_1} &= \lie{g}^2  \\
                \vecspan{x_4, x_2, x_1} &= \lie{g}^2 + Z(\lie{g}) \\
                \vecspan{x_3, x_4, x_2, x_1}
                    &= \{ x \in \lie{g} \mid \dim [x, \lie{g}] \leq 1 \}
            \end{align*}
        \item If $\lie{g} = \lie{n}_4 \semisum \R$ and $[x_5, x_4] = 0$:
            \begin{align*}
                \vecspan{x_1} &= \lie{g}^3  \\
                \vecspan{x_2, x_1} &= \lie{g}^2  \\
                \vecspan{x_5, x_2, x_1}
                    &= \{x \in \lie{g} \mid [x, \lie{g}] \subseteq \vecspan{x_1}\} \\
                \vecspan{x_3, x_5, x_2, x_1}
                    &= \{ x \in \lie{g} \mid [x, \vecspan{x_2, x_1}] = [x, \vecspan{x_5, x_2, x_1}] \}
            \end{align*}
        \item If $\lie{g} = \lie{n}_4 \semisum \R$ and $[x_5, x_4] = x_3$:
            \begin{align*}
                \vecspan{x_1} &= \lie{g}^4  \\
                \vecspan{x_2, x_1} &= \lie{g}^3  \\
                \vecspan{x_3, x_2, x_1} &= \lie{g}^2  \\
                \vecspan{x_5, x_3, x_2, x_1}
                    &= \{ x \in \lie{g} \mid \dim [x, \lie{g}] \leq 2 \} \qedhere
            \end{align*}
    \end{itemize}

    \paragraph{Case 2: $G$ is not nilpotent.}
    The descriptions of these Lie algebras
    are explicit enough that the eigenvalues can be verified
    to be real by inspection. So again it suffices to show that in each case,
    $\op{Aut} G = \op{Aut} \lie{g}$
    contains no nontrivial connected compact subgroups.
    By Lemma \ref{lemma:compacts_in_aut}, it suffices to show that
    $(\op{Aut} \lie{g})^0$ is upper-triangular in some basis.

    In each case, an automorphism of $\lie{g}$ preserves the nilradical $\lie{n}$
    and the decomposition of $\lie{n}$ into generalized eigenspaces of $\lie{g}/\lie{n}$,
    up to any reordering and scaling.
    Each generalized eigenspace has a natural flag---the
    filtration by the rank of generalized eigenvectors---which
    is also preserved. Except in two cases, this is enough data
    to make $(\op{Aut} \lie{g})^0$ upper-triangular.
    The two cases and their additional data are as follows.
    \begin{itemize}
        \item In the $(\R \oplus \lie{n}_3) \semisum \R$ geometries, the flag
            $0 \subset \vecspan{x_1} \subset \vecspan{y, x_1}$ in the $0$-eigenspace
            is preserved since $\vecspan{x_1} = [\lie{n}, \lie{n}]$.
        \item In $\R^3 \semisum \R^2$, the subset of $\R^2 \cong \lie{g}/\lie{n}$
            consisting of points that act with a zero eigenvalue is preserved.
            This is a set of three concurrent lines. \qedhere
    \end{itemize}
\end{proof}

\begin{prop}[\textbf{Claim of correctness of Figure \ref{fig:solvable_flowchart}}]
    \label{prop:solvable_distinctness}
    All of the maximal geometries named above are distinct,
    with the exception that some of the geometries
    $\R^4 \semisum \R$ with $4$ distinct real eigenvalues
    may coincide with each other.
\end{prop}
\begin{proof}
    Isomorphic geometries have isomorphic transformation groups;
    so it suffices to show that the corresponding Lie algebras
    are mutually non-isomorphic.

    Referring to the calculation in Prop.~\ref{prop:solvable_maximality},
    the nilpotent algebras are distinguished by whether they admit $4$-dimensional
    abelian ideals and the number of nonzero terms in their lower central series.

    The non-nilpotent algebras can be subdivided
    according to the isomorphism type of
    their nilradicals.
    \begin{itemize}
        \item Only $\R^3 \semisum \R^2$ has nilradical $\R^3$.
        \item The two $(\R \oplus \lie{n}_3) \semisum \R$ geometries
            are distinguished from each other by the dimensions of their
            centers---which are $\vecspan{x_1}$ if $[z,y] = x_1$,
            and $\vecspan{x_1, y}$ if $[z,y] = 0$).
        \item The non-nilpotent algebras $\lie{g} = \R^4 \semisum \R$
            can be distinguished from each other by the Jordan blocks
            by which $\R \cong \lie{g}/\R^4$ acts on the nilradical $\R^4$---up
            to a scale factor, to account for the ability to rescale $\R$.
            This distinguishes all but the case when one action of $\R$ on $\R^4$
            has $4$ distinct real eigenvalues that are a constant multiple
            of those of another action. \qedhere
    \end{itemize}
\end{proof}

\bibliographystyle{amsalpha}
\bibliography{main}

\end{document}